\documentclass[a4paper]{amsart}

\usepackage{amsthm,amssymb,amsmath}
\usepackage[colorlinks=true,citecolor=blue]{hyperref}
\usepackage{fullpage}
\usepackage{multirow}
\usepackage{graphicx}
\usepackage[all]{xy}

\newcommand{\mmid}{\mathop{\|}}
\newcommand{\lddot}{\mathop{..}}
\DeclareMathOperator{\limProb}{lim\,Prob}

    \newcommand{\BA}{{\mathbb {A}}} 
    \newcommand{\BC}{{\mathbb {C}}} 
     \newcommand{\BF}{{\mathbb {F}}}

     \newcommand{\BP}{{\mathbb {P}}}
    \newcommand{\BQ}{{\mathbb {Q}}} \newcommand{\BR}{{\mathbb {R}}}
     \newcommand{\BT}{{\mathbb {T}}}

     \newcommand{\BZ}{{\mathbb {Z}}}

    \newcommand{\CE}{{\mathcal {E}}} 
     \newcommand{\CH}{{\mathcal {H}}}
     
     \newcommand{\CL}{{\mathcal {L}}}
     
    \newcommand{\CO}{{\mathcal {O}}}

     \newcommand{\RT}{{\mathrm {T}}}

    \newcommand{\fm}{{\mathfrak{m}}} 
     \newcommand{\fp}{{\mathfrak{p}}}

     \newcommand{\fX}{{\mathfrak{X}}}
     \newcommand{\fZ}{{\mathfrak{Z}}}

    \newcommand{\Aut}{{\mathrm{Aut}}}

    \DeclareMathOperator{\corank}{corank}
    
    \newcommand{\Cl}{{\mathrm{Cl}}}

    \newcommand{\diag}{{\mathrm{diag}}}
     
    \newcommand{\End}{{\mathrm{End}}}

    \newcommand{\Gal}{{\mathrm{Gal}}}

    \newcommand{\id}{{\mathrm{id}}}\renewcommand{\Im}{{\mathrm{Im}}}

    \newcommand{\Ker}{{\mathrm{Ker}}}

    \newcommand{\loc}{{\mathrm{loc}}}

    \newcommand{\ord}{{\mathrm{ord}}} \DeclareMathOperator{\rank}{rank}

    \newcommand{\LHS}{{\mathrm{LHS}}}
    
    \renewcommand{\mod}{\ \mathrm{mod}\ }
    
    \newcommand{\RHS}{{\mathrm{RHS}}}

    \newcommand{\Rows}{{\mathrm{Rows}}}
    \newcommand{\Sel}{{\mathrm{Sel}}} 
    
    \newcommand{\supp}{{\mathrm{supp}}}
    
    \newcommand{\sign}{{\mathrm{sign}}}

    \newcommand{\tor}{{\mathrm{tor}}}

\newcommand{\matrixx}[4]{\begin{pmatrix}
#1 & #2 \\ #3 & #4
\end{pmatrix} }        

    \font\cyr=wncyr10

    \newcommand{\Sha}{\hbox{\cyr X}}
    \newcommand{\wh}{\widehat}
    
    \newcommand{\pair}[1]{\langle {#1} \rangle}

    \newcommand{\ov}{\overline}

    \newcommand{\sk}{\medskip}
    \newcommand{\lra}{\longrightarrow}
    \newcommand{\ra}{\rightarrow} 
    \newcommand{\bs}{\backslash}
    
    \newcommand{\s}{\sk\noindent}

    \theoremstyle{plain}
    \newtheorem{thm}{Theorem}[section] 
    \newtheorem{lem}[thm]{Lemma}  \newtheorem{prop}[thm]{Proposition}
     
     \newtheorem{lem-defn}[thm]{Lemma-Definition}

\theoremstyle{remark} \newtheorem{remark}[thm]{Remark}
\theoremstyle{remark} 
\theoremstyle{remark} 
\theoremstyle{remark} 

    \numberwithin{equation}{section}
    \newcommand{\lrd}[2]{\left(\frac{#1}{#2}\right)}
     \newcommand{\lrdd}[2]{\left[\frac{#1}{#2}\right]}

\begin{document}

\title{Quadratic twists of tiling number elliptic curves}
\thanks{Keqin Feng is supported by NSFC with no.12031011.}

\author{Keqin Feng}
\address{Department of Mathematical Sciences,
Tsinghua University, Beijing, China}
\email{fengkq@tsinghua.edu.cn}

\author{Qiuyue Liu}
\address{Center for Applied Mathematics, Tianjin University,
No. 92 Weijin Road, Tianjin 300072, China}
\email{17864309562@163.com}

\author{Jinzhao Pan}
\address{Institut f\"ur Algebra, Zahlentheorie und Diskrete Mathematik,
Leibniz Universit\"at Hannover, Hannover, Germany}
\email{jinzhao.pan@math.uni-hannover.de}

\author{Ye Tian}
\address{Morningside Center of Mathematics,
Chinese Academy of Sciences, Beijing, China}
\email{ytian@math.ac.cn}

\begin{abstract}
A positive integer $n$ is called a tiling number if the equilateral triangle can be dissected into $nk^2$ congruent triangles for some integer $k$.
An integer $n>3$ is tiling number if and only if at least one of the elliptic curves
$E^{(\pm n)}:\pm ny^2=x(x-1)(x+3)$ has positive Mordell-Weil rank.
Let $A$ denote one of the two curves.
In this paper, using Waldspurger formula and an induction method,
for $n\equiv 3,7\mod 24$ positive square-free,
as well as some other residue classes,
we express the parity of analytic Sha of $A$ in terms of the
genus number $g(m):=\#2\Cl(\BQ(\sqrt{-m}))$ as $m$ runs over factors of $n$.
Together with $2$-descent method which express $\dim_{\BF_2}\Sel_2(A/\BQ)/A[2]$
in terms of the corank of a matrix of $\BF_2$-coefficients,
we show that for $n\equiv 3,7\mod 24$ positive square-free,
the analytic Sha of $A$ being odd is equivalent to that
$\Sel_2(A/\BQ)/A[2]$ being trivial, as predicted by the BSD conjecture.

We also show that,
among the residue classes $3$, resp.~$7\mod 24$,
the subset of $n$ such that both of $E^{(n)}$ and $E^{(-n)}$ have analytic Sha odd
is of limit density $0.288\cdots$ and
$0.144\cdots$, respectively, in particular, they are non-tiling numbers.
This exhibits two new phenomena on tiling number elliptic curves:
firstly, the limit density is different from
the general phenomenon on elliptic curves
predicted by Bhargava-Kane-Lenstra-Poonen-Rains;
secondly, the joint distribution has different behavior among
different residue classes.
\end{abstract}

\maketitle

\tableofcontents

\section{Introduction and main results}

A positive integer $n$ is called a {\em tiling number} if the equilateral triangle can be dissected into $nk^2$ congruent triangles for some integer $k$.
To determine whether a given integer is a tiling number,
one only needs to consider square-free integers.
It is known that a square-free integer $n>3$ is a
tiling number if and only if at least one of the elliptic curves (called tiling number curves)
$$
\pm ny^2=x(x-1)(x+3)
$$
has positive Mordell-Weil rank \cite{Lac}.
Tiling number curves form  an example of quadratic twist family of elliptic curves.
We study the $2$-adic aspect of arithmetic of this family
to compare with the family of congruent number curves
\cite{Tian} \cite{TYZ}.

In general, for a quadratic twist family of elliptic curves, we define an equivalence relation to divide the family into finitely many classes. The distribution of $2$-Selmer among these equivalence classes can be very different for different quadratic families. For example, among all positive integers $n\equiv 1\mod 12$, the curves $-ny^2=x(x-1)(x+3)$ always have non-trivial 2-Selmer modulo torsion (see Proposition \ref{descent non-trivial example}),
but it is not the case for congruent number elliptic curves $ny^2=x^3-x$:
for any congruence class of $n$ such that the curve is of sign $+1$,
there are always positive density of $n$ such that
the curve has trivial 2-Selmer modulo torsion \cite{HB2} \cite{SD} \cite{Kane}
\cite{Smith}.
For more comparison with congruent number curves, see Remark \ref{comparison to CNP}.

It would be interesting to study the distribution of $2$-Selmer groups for tiling number curves.
In this paper,  we mainly  focus on the minimal $2$-Selmer case,
emphasizing the difference from congruent number curves.
The main result of this paper is

\smallskip
\s{\bf Theorem A}. {\em Let $\fX$ be the family of elliptic curves $ny^2=x(x-1)(x+3)$ (resp. $-ny^2=x(x-1)(x+3)$) with positive square-free $n\equiv 3\mod 24$. For $A\in \fX$, the followings are equivalent:
\begin{itemize}\item[(1)] The $2$-Selmer group of $A$ modulo torsion is trivial;
\item[(2)] $A$ has analytic rank $0$ and ananytic Sha odd;
\item[(3)] $\BQ(\sqrt{-n})$ has no ideal classes of exact order $4$;
\item[(4)] let $Q(x, y, z)=x^2+3y^2+4z^2$ (resp. $6x^2+y^2+2z^2$),  $\ord_2(r_Q(n/3))$ equals the number of prime factors of $n/3$, where
$$r_Q(m):=\# \left\{ (x, y, z)\in \BZ^3\ \Big|\ Q(x, y, 2z)=m\right\} -\frac{1}{2}\# \left\{ (x, y, z)\in \BZ^3\ \Big|\ Q(x, y, z)=m\right\}.$$

\end{itemize}
Moreover,  the subset of $\fX$ of curves $A$ satisfying the above equivalence conditions  has limit density $$\delta:=\prod_{i=1}^\infty(1-2^{-i})=0.288788\cdots.$$}

\smallskip

The equivalence between (1) and (3) is by $2$-descent (Proposition \ref{descent main 3}).
To have the equivalence between (2) and (3) (Theorem \ref{analytic sha main}),
we employ the Waldspurger formula and the induction method introduced by one of the authors in the study of congruent number problem \cite{Tian} \cite{TYZ}.
However, in our current situation the curves are non-CM which cause extra difficulty. The equivalence of (2) and (4)
(Theorem \ref{ternary 7})
follows from Waldspurger's work.
The limit density is Theorem \ref{positive density 3}.
Note that this number $\delta$ is different from
the general phenomenon on elliptic curves
predicted by Bhargava-Kane-Lenstra-Poonen-Rains \cite{BKLPR},
also different from
the results on quadratic twists of elliptic curves
by Heath-Brown \cite{HB2}, Swinnerton-Dyer \cite{SD} and Kane \cite{Kane},
which agree with BKLPR's prediction.

\medskip

A statement similar to Theorem A (with item (3) more complicated) holds for the case $n\equiv 7\mod 24$
((1)$\Leftrightarrow$(3) is Proposition \ref{descent main}
and Theorem \ref{comparison 7},
(2)$\Leftrightarrow$(3) is Theorem \ref{analytic sha main},
(2)$\Leftrightarrow$(4) is Theorem \ref{ternary 7}).
By analyzing these phenomena we obtain some property on the joint distribution behavior  of $2$-Selmer groups among classes for the tiling number curves.
It follows from Theorem A that the two classes $\pm ny^2=x(x-1)(x+3)$, $n\equiv 3\mod 24$ have a trivial joint.
However, the case of $n\equiv 7\mod 24$ is different.

\s{\bf Theorem B} (Theorem \ref{positive density 7}).
{\em Among all positive square-free integers $n\equiv 7\mod 24$,
the subset of $n$ such that
$ny^2=x(x-1)(x+3)$ (resp. $-ny^2=x(x-1)(x+3)$)
have trivial $2$-Selmer modulo torsion
has limit density $\delta=0.288788\cdots$. Moreover, the subset of $n$ such that
both of $\pm ny^2=x(x-1)(x+3)$
have trivial $2$-Selmer modulo torsion
has limit density $\delta/2=0.144394\cdots$.}

\smallskip

Let $\Sigma=\{\infty,2,3\}$.
We call two square-free integers $n, n'$ $\Sigma$-equivalent if $n/n' \in (\BQ_v^\times)^2$ for all $v\in \Sigma$. This defines an equivalence relation among all square-free integers.
The sign of $E^{(n)}:ny^2=x(x-1)(x+3)$
only depends on the $\Sigma$-equivalence
class $[n]$ containing $n$ and we call it the sign of the class. There are $64$ classes and half of them have sign $+1$.  Among the $32$ classes of sign $+1$,
by $2$-descent in \S\ref{s:descent} and matrix analysis \S\ref{s:positive density},
one can similarly obtain
(see Remark \ref{comparison to CNP})
\begin{itemize}
\item
there are $27$ classes, for each the limit density of $2$-Selmer modulo torsion being trivial is $\delta$;
\item
there are $3$ classes, $[1], [33], [57]$, for each the limit density of $2$-Selmer modulo torsion being trivial is $\frac{4}{3}\delta$;
\item
there are $2$ classes, $[-1], [-13]$,  for each the limit density of $2$-Selmer modulo torsion is $0$.
\end{itemize}
Hence one can obtain the following result for the limit density of $2$-Semler
being trivial for tiling number curves (noting that
each class have different relative size in the set of square-free integers).

\s{\bf Theorem C}. {\em Among all square-free integers $n$ with $\sign(E^{(n)})=+1$, the subset of $n$ such that $E^{(n)}$ has trivial $2$-Selmer modulo torsion has limit density
$$\frac{131}{144}\delta =0.262716\cdots.$$}

The Birch and Swinnerton-Dyer conjecture predicts that, for all $n$
in all the above classes,
the $2$-Selmer of $E^{(n)}$ modulo torsion is trivial
is equivalent to that the analytic Sha of $E^{(n)}$ is odd.
From Theorem \ref{analytic sha main}
it's expected that our method without large modification
should be able to prove this assertion for all these classes
(not only for $[\pm 3]$ and $[\pm 7]$),
except for the class $[1]$ where the analytic Sha of $E^{(n)}$
and $E^{(-n)}$ cannot be distinguished easily,
and classes $[2],[14],[26],[38]$ where the analytic Sha of $E^{(n)}$
cannot be expressed by only the genus numbers
$g(n):=\#2\Cl(\BQ(\sqrt{-n}))$.

\s{\bf Notations}.
Let $E:y^2=x^3+ax+b$ be a fixed elliptic curve over $\BQ$.
For $n\in\BQ^\times/(\BQ^\times)^2$, define the quadratic twist
$E^{(n)}$ to be $ny^2=x^3+ax+b$.

For an element $n$ of $\BQ^\times/(\BQ^\times)^2$,
let $S(n)$ denote
the $2$-Selmer group of $E^{(n)}$ modulo torsion:
$$
S(n):=\Sel_2(E^{(n)}/\BQ)/E^{(n)}(\BQ)_\tor,
$$
and $s(n)$ its dimension over $\BF_2$.
We also define
\begin{equation}
\label{analytic sha definition}
\CL(n):=\frac{L(E^{(n)},1)}{\Omega(E^{(n)})}\cdot \left(\frac{\prod_p c_p(E^{(n)})}{\#E^{(n)}(\BQ)_\tor^2}\right)^{-1},
\end{equation}
where $\Omega(E^{(n)})$ is
the period of $E^{(n)}$, and $c_p(E^{(n)})$ is the Tamagawa number of $E^{(n)}$
at $p$.
In other words, $\CL(n)=0$ if $L(E^{(n)},1)=0$,
and $\CL(n)$ is the analytic Sha predicted by the Birch and Swinnerton-Dyer conjecture
if $L(E^{(n)},1)\neq 0$.

Let $\fX$ be a subset of square-free integers,
and let $\fZ$ be a subset of $\fX$.
Then the limit density of $\fZ$ in $\fX$ is defined to be
$$
\limProb(\fZ\mid\fX)
:=\lim_{r\to\infty}\lim_{N\to\infty}
\BP\left(n\in\fZ~\middle|\begin{array}{l}
n\in\fX\text{ with} \\
\text{exactly }r\text{ prime factors and }|n|<N
\end{array}
\right).
$$

\section{Toric Periods and $L$-values}
\label{s:toric-periods}

First let's recall some notations and fact about toric periods and
Waldspurger formula in \cite{CST}.

Let $E$ be an elliptic curve over $\BQ$ of conductor $N$,
$\Sigma$ the places of $\BQ$ consisting of $2,\infty$ and bad places for $E$,
$\phi=\sum_{n=1}^\infty a_n q^n\in S_2(\Gamma_0(N))$ the newform associated to $E$.
Let $K$ be an imaginary quadratic field of discriminant $D$, $\Cl(K)$ be its ideal class group and $\chi:\Cl(K) \ra \{\pm 1\}$ be an unramified quadratic character.
Let $\eta$ be the quadratic character associated to the extension $K/\BQ$.
Let $\Sigma^-\subset\Sigma$ be the set of places such that
$$\epsilon_v(E\times\chi)\cdot \chi_v\eta_v(-1)=-1.$$
Here $\epsilon_v(E\times\chi)$ is the local root number at $v$ of the
Rankin-Selberg $L$-series
$L(E\times\chi,s)$.
Note that $\infty\in \Sigma^-$ and for any finite place $v\in \Sigma^-$, $v$ is not split in $K$.
Moreover, if $p$ is inert in $K$,
or if $p\mmid N$ is ramified in $K$, then there are simple criteria for $p\in\Sigma^-$:
\begin{itemize}
\item
if $p$ is inert in $K$, then
$p\in \Sigma^-$ if and only if $\ord_p(N)$ is odd;
\item
if $p\mmid N$ is ramified in $K$, then $p\in \Sigma^-$ if and only if $\chi([\fp])=a_p$, where $\fp$ is the unique prime of $K$ above $p$.
\end{itemize}

Assume that $\Sigma^-$ has even cardinality, equivalently, the sign of the functional equation of $L(E\times\chi,s)$ is $+1$. Let $B$ be the definite quaternion algebra over $\BQ$ ramified exactly at places in $\Sigma^-$. Note that
there exists an embedding of $K$ into $B$, which we fix once and for all, and then view $K$ as a $\BQ$-subalgebra of $B$.  Let $R\subset B$ be an order of discriminant $N$ with $R\cap K=\CO_K$. Such an order exists and is unique up to conjugation by $\wh{K}^\times$. Here, for an abelian group $M$, we define $\wh{M}=M\otimes_\BZ \wh{\BZ}$ where $\wh{\BZ}=\prod_p \BZ_p$ with $p$ running over all primes.
The Shimura set
$$X:=B^\times \bs \wh{B}^\times /\wh{R}^\times$$
is a finite set.
For each $[g]\in X$, let $w_g$
be the order of the finite group $(B^\times \cap g\wh{R}^\times g^{-1})/\{\pm 1\}$.
Let $\BZ[X]$ denote the set of $\BZ$-valued functions on $X$,
and define a bilinear pairing on $\BZ[X]$ by
$$
\pair{f_1, f_2}
:=\sum_{[g]\in X}w_g^{-1}f_1(g)f_2(g).
$$
The pairing is positive definite on $\BR[X]:=\BZ[X]\otimes_\BZ \BR$ and has a natural Hermitian extension to $\BC[X]:=\BZ[X]\otimes_\BZ \BC$.
Let $\BC[X]^0\subset\BC[X]$ be the orthogonal complement of
the functions on $X$ which factor through the reduced norm map
$\widehat B^\times\to\widehat\BQ^\times$.
Since the constant map $X\to\BC$, $[g]\mapsto 1$ factors
through the reduced norm map,
for any $f\in\BC[X]^0$ we have
$\deg(f):=\sum_{[g]\in X}w_g^{-1}f(g)=0$.

For each prime $p\nmid N$,
fix an isomorphism $R_p\cong M_2(\BZ_p)$,
define the
Hecke correspondence $T_p$ on $X$ by $T_p [g]=\sum_{h_p} [g^{(p)}h_p]$
if $g=(g_v)\in \wh{B}^\times$, where $g^{(p)}$ is the $p$-off part of $g$, and
$h_p$ runs over $\left\{g_p\left(\begin{smallmatrix}
p&i \\ &1
\end{smallmatrix}\right)\mid i\in\BZ/p\BZ\right\}\sqcup
\left\{g_p\left(\begin{smallmatrix}
1& \\ &p
\end{smallmatrix}\right)\right\}$.
By pullback of Hecke correspondence we define the Hecke action
$T_p$ on $\BZ[X]$.
Let $\BT\subset \End(\BZ[X])$ be the commutative algebra generated over $\BZ$ by all $T_p$ for $p\nmid N$. Then $\BZ[X]^0:=\BC[X]^0\cap\BZ[X]$ is stable under the action of $\BT$.  The Hecke operators $T_p$ are self-adjoint with respect to the Hermitian pairing, thus
$$
\BC[X]^0=\bigoplus_i \BC f_i
$$
where $\{f_i\}$ is an orthogonal basis
of $\BC[X]^0$
consists of eigenvectors for $\BT$.

If $v\mid(N, D)$, then $K_v^\times$ normalizes $R_v^\times$ and thus acts on $X$ by right multiplication. The subspace
$V=V(E,\chi)$ of $\BC[X]^0$ where $T_p$ acts as $a_p$ for all $p\nmid N$ and $K_v^\times$ acts via $\chi_v$ for all $v\mid(N, D)$,
called the space of test vectors \cite{GP} for $(E,\chi)$,
is one-dimensional. Moreover,
since the coefficients of $T_p$-actions and the values of $\chi_v$ are integral,
$V$ has a basis $f$ viewed as a function $f: X\ra \BZ$ and we may assume that $f$ is primitive in the sense the image of $f$ generates $\BZ$.
Such $f$ is unique up to $\pm 1$ and is called a primitive test vector
for $(E,\chi)$.
The inclusion $K\ra B$ induces a map $\iota: \Cl(K)\ra X$ using which we define an element in $\BC$, called the toric period associated to $(E, \chi)$:
$$P_\chi(f):=\sum_{t\in \Cl(K)} \chi(t) f(\iota(t)).$$
Then we have the Waldspurger formula
\begin{thm} \label{Wald} Let $E$ and $\chi$ be as above. Then
$$L(E\times\chi,1)=2^{-\mu(N, D)} \cdot \frac{8\pi^2 (\phi, \phi)_{\Gamma_0(N)}}{u^2 \sqrt{|D|}} \cdot
\frac{|P_\chi(f)|^2}{\pair{f, f}}.$$
Here $\mu(N, D)$ is the number of common prime factors of $N$ and $D$, $u=[\CO_K^\times: \BZ^\times]$, $(\phi, \phi)_{\Gamma_0(N)}$ is
the Petersson norm of $\phi$:
 $$(\phi, \phi)_{\Gamma_0(N)}=\int_{\Gamma_0(N)\bs \CH} |\phi(z)|^2 dx dy, \qquad z=x+iy.$$
\end{thm}

Let $H_0$ be the genus field of $K$,
that is, the maximal unramified abelian extension of $K$
of exponent $2$.
Then $\Cl(K)/2\Cl(K)\cong\Gal(H_0/K)
\cong(\BZ/2\BZ)^{\mu(D)-1}$,
where $\mu(D)$ is the number of distinct prime factors of $D$.
Write $K=\BQ(\sqrt{-n})$ where $n$ is a positive square-free integer,
then any unramified quadratic extension over $K$
is a subfield of $H_0$,
and has form $K(\sqrt{-d})$ with $d\mid n$ positive; when $n\equiv 3\mod 4$ (resp.~$2\mid n$) it's moreover required that
$d\equiv 3\mod 4$
(resp.~$d\equiv 3\mod 4$ or $d\equiv n\mod 8$).
Denote by $\chi_d$ the corresponding unramified quadratic character
of $K$.
In particular, taking $d=n$, the $\chi_n=1_K$ is the trivial character over $K$.

For simplicity assume that $E[2]\subset E(\BQ)$.
Waldspurger formula allows us to relate the toric period $P_\chi(f)$ of an unramified
quadratic character $\chi=\chi_d$ with analytic Sha $\CL(n/d)$ and $\CL(-d)$
of elliptic curves occurring in the decomposition $L(E\times\chi,s)=L(E^{(n/d)},s)L(E^{(-d)},s)$.

\begin{prop}
\label{period0}
Let $\Sigma$ be the set of places of $\BQ$ containing $2,\infty$ and bad places of $E$.
Denote by $[n]$ and $[d]$ the $\Sigma$-equivalence classes containing $n$
and $d$, respectively.
Then for all $n$ and $d$
as above, the isomorphism class of $K_v$
for all $v\mid N$ only depends on $[n]$
and the following results hold:
\begin{itemize}
\item[(i)]
The root numbers of $E^{(n/d)}$ and $E^{(-d)}$
only depend
on $[n]$ and $[d]$.
\end{itemize}
Suppose $n$ and $d$ are such that
the root numbers of $E^{(n/d)}$ and $E^{(-d)}$ are both $+1$.
Then
\begin{itemize}
\item[(ii)]
The set $\Sigma^-\subset\Sigma$ such that
$\epsilon_v(E\times\chi)\cdot \chi_v\eta_v(-1)=-1$ only depends
on $[n]$ and $[d]$.
\item[(iii)]
Let $(n,d)$ and $(n',d')$ be such that $[n]=[n']$
and their associated $\Sigma^-$ equal.
Choose any embeddings
$\iota:K\to B$ and $\iota':K'\to B$.
Choose any order $R$ of $B$ of discriminant $N$ with
$R\cap\iota(K)=\iota(\CO_K)$.
Then one may find an order $R'$ of $B$ of discriminant $N$
with $R'\cap\iota'(K')=\iota'(\CO_{K'})$,
such that $R$ and $R'$ are locally conjugate to each other,
and the induced isomorphism $X\xrightarrow\sim X'$
preserves Hecke action and $K_v^\times$-action for all $v\mid(N,D)$,
here $D$ is the discriminant of $K$.
Therefore it preserves primitive test vectors
for $(E,\chi)$ and $(E,\chi')$
in the case that $\chi$ and $\chi'$ have the same local components at $v$
for all $v\mid(N,D)$,
in particular, primitive test vector only depends
on $[n]$ and $[d]$.
\item[(iv)]
Let $f$ be a primitive test vector for $(E,\chi)$
in the sense that $f\in V(E,\chi)$, $f:X\to\BZ$
whose image generates $\BZ$.
Then
$$
\left|\frac{P_\chi(f)}{2^{\mu(D)-1}}\right|^2=C_{n,d}\cdot
\CL(n/d)\CL(-d)\cdot T_{n,d},
$$
where $C_{n,d}$ is an explicit
non-zero rational constant which only depends
on $[n]$ and $[d]$,
$$
T_{n,d}:=[\CO_K^\times:\BZ^\times]^2\left(\frac{4}{\#E^{(n/d)}(\BQ)_\tor}\right)^2
\left(\frac{4}{\#E^{(-d)}(\BQ)_\tor}\right)^2.
$$
\end{itemize}
\end{prop}

\begin{proof}
(i) This comes from Lemma \ref{quad twist arith values}.

(ii) For each $v\mid N$, we have
$\epsilon_v(E\times\chi)=\epsilon_v(E^{(n/d)})
\epsilon_v(E^{(-d)})$, and by Lemma \ref{quad twist arith values},
it only depends on $[n]$ and $[d]$.
The $\chi_v\eta_v(-1)$ also only depends on $[n]$ and $[d]$.

(iii) For each $v\mid N$ we fix an isomorphism
$K_v\cong K'_v$,
and denote $\iota_v:K_v\to B_v$
and $\iota_v':K_v\cong K'_v\to B_v$
the embeddings induced by $\iota$ and $\iota'$,
then there exists an element $\gamma_v\in B_v^\times$
such that $\iota_v'=\gamma_v\iota_v\gamma_v^{-1}$.
There exist maximal orders $\CO_B$
and $\CO_B'$
of $B$ such that $\CO_B\cap\iota(K)=\iota(\CO_K)$
and $\CO_B'\cap\iota'(K')=\iota'(\CO_{K'})$.
The $\CO_B$ and $\CO_B'$ are locally conjugate to each other,
say for each $v\nmid N$,
$\gamma_v\in B_v^\times$ is such that $\CO_{B,v}'=\gamma_v\CO_{B,v}\gamma_v^{-1}$;
note that $\CO_{B,v}'=\CO_{B,v}$ and $\gamma_v=1$ for all but finitely
many $v$.

One of the $R$ can be constructed as follows:
let $\fm\subset\CO_K$ be an integral ideal of norm
$N/\prod_{p\in\Sigma^-\setminus\{\infty\}}p$, then let $R=\iota(\CO_K)+\iota(\fm)\CO_B$.
Conversely, any such $R$ comes from this way, namely there exists a maximal
order $\CO_B$ of $B$ satisfying $\CO_B\cap\iota(K)=\iota(\CO_K)$,
and an integral ideal $\fm\subset\CO_K$, such that
$R=\iota(\CO_K)+\iota(\fm)\CO_B$.

For each $v\mid N$, let $\alpha_v\in\CO_{K,v}$ be an element
such that $\alpha_v\CO_{K,v}=\fm\CO_{K,v}$,
then $R_v=\iota_v(\CO_{K,v})+\iota_v(\alpha_v)\CO_{B,v}$,
and for each $v\nmid N$ we have $R_v=\CO_{B,v}$.
Now for each $v\mid N$ we construct
$R_v'$ by $R_v'=\gamma_vR_v\gamma_v^{-1}
=\iota_v'(\CO_{K,v})+\iota_v'(\alpha_v)\cdot\gamma_v\CO_{B,v}\gamma_v^{-1}$,
and construct $R'$ by $R'=B\cap(\prod_{v\mid N}R_v'\times\prod_{v\nmid N}\CO_{B,v}')$,
then $\widehat R'=\gamma\widehat R\gamma^{-1}$ with $\gamma=(\gamma_v)\in\widehat B^\times$
where $\gamma_v$ is described as before.
Hence $R'$ is locally conjugate to $R$,
is of discriminant $N$, and $R'\cap\iota'(K')=\iota'(\CO_{K'})$.
The isomorphism $X\xrightarrow\sim X'$ is given by $[g]\mapsto[g\gamma^{-1}]$,
it's easy to see that it preserves Hecke action and $K_v^\times$-action.

(iv) By Waldspurger formula (Theorem \ref{Wald}),
the definition of analytic Sha \eqref{analytic sha definition},
and the relation of Petersson inner product
and periods of elliptic curve \eqref{period of mod form and ec}, we have
\begin{multline*}
|P_\chi(f)|^2
=2^{\mu(N, D)}\langle f,f\rangle\cdot
\frac{c_E^2}{\deg\varphi}\cdot\frac{u^2\sqrt{D}}{\Omega(E)\Omega^-(E)}
\cdot
\Omega(E^{(n/d)})\Omega(E^{(-d)}) \\
{}\cdot
\frac{\prod_p c_p(E^{(n/d)})}{\#E^{(n/d)}(\BQ)_\tor^2}\cdot
\frac{\prod_p c_p(E^{(-d)})}{\#E^{(-d)}(\BQ)_\tor^2}\cdot
\CL(n/d)\CL(-d),
\end{multline*}
here $\varphi:X_0(N)\to E$ is a modular parametrization of $E$
of minimal degree.
The relation of periods of elliptic curves under quadratic twist
is given by Lemma \ref{quad twist arith values}, which induces
$$
\frac{u^2\sqrt{D}}{\Omega(E)\Omega^-(E)}\cdot
\Omega(E^{(n/d)})\Omega(E^{(-d)})
=c_\infty(E)u^2\sqrt{-D/n}\cdot
u^{(n/d)}u^{(-d)},
$$
here $u^{(n/d)}$ and $u^{(-d)}$ compares the minimal Weierstrass equations
under quadratic twist.
Therefore the constant $C_{n,d}$ equals
\begin{multline}
\label{e:Cnd}
C_{n,d}=4^{\mu(n)-\mu(D)-3}
2^{\mu(N,D)}\langle f,f\rangle\cdot\frac{c_E^2c_\infty(E)}{\deg\varphi}
\cdot\sqrt{-D/n}\cdot u^{(n/d)}u^{(-d)} \\
\cdot
4^{-\mu(n/d)}\prod_p c_p(E^{(n/d)})\cdot
4^{-\mu(d)}\prod_p c_p(E^{(-d)}).
\end{multline}
Note that $\mu(N,D)$ only depends on $[n]$,
$\langle f,f\rangle$ only depends on $[n]$ and $[d]$,
$\mu(n)-\mu(D)\in\{0,-1\}$ and
$\sqrt{-D/n}\in\{1,2\}$ only depend on $n\mod 4$ (hence $[n]$),
and by Lemma \ref{quad twist arith values},
$u^{(n/d)}$ and $4^{-\mu(n/d)}\prod_p c_p(E^{(n/d)})$
(resp.~$u^{(-d)}$ and $4^{-\mu(d)}\prod_p c_p(E^{(-d)})$)
only depends on $[n/d]$
(resp.~$[d]$).
So the desired result follows.
\end{proof}

\begin{remark}
For all but finitely many square-free integer $n$,
$E^{(n)}(\BQ)_\tor=E^{(n)}(\BQ)[2]$,
so $T_{n,d}=1$ in most cases.
\end{remark}

\subsection{Test functions for Waldspurger formula}

We apply Proposition \ref{period0} and \eqref{e:Cnd}
to the elliptic curve
$$
E: y^2=x(x-1)(x+3),
$$
as well as its quadratic twist $E^{(-1)}$.
The conductor of $E$ and $E^{(-1)}$ are $24$ and $48$, respectively,
so in both cases $\Sigma=\{\infty,2,3\}$.
All of $E^{(\pm 1)}$ and $E^{(\pm 3)}$ have root numbers $+1$
and $\CL(\pm 1)=\CL(\pm 3)=1$.

We introduce the following notation:
a pair $(n,d)$ is called \emph{admissible} for $E$,
if the root numbers of $E$,
$E^{(-n)}$, $E^{(n/d)}$ and $E^{(-d)}$ are all $+1$,
and such that $(E,\chi_d)$ and $(E,1_K)$ have same primitive test vectors.

\begin{thm}
\label{period1}
Let $E:y^2=x(x-1)(x+3)$ and $n\neq 1,3$ be a positive square-free integer.

{\rm(i)}
If $(n,d)$ is admissible for $E$, then
the primitive test vector $f$ for $(E,\chi_d)$
only takes values in odd integers, and
$$
\left|\frac{P_\chi(f)}{2^{\mu(D)-1}}\right|^2
=C_{n,d}\cdot\CL(n/d)\CL(-d)\cdot T_{n,d}
$$
where
$$
C_{n,d}=\begin{cases}
1,&\text{if }n\equiv 1\mod 12\text{ or }(n,d)\equiv(7,11)\mod 24, \\
4,&\text{otherwise},
\end{cases}
\quad\text{and}\quad
T_{n,d}=\begin{cases}
1/4, &\text{if }d=n, \\
1,&\text{otherwise}.
\end{cases}
$$

{\rm(ii)}
If $(n,d)$ is admissible for $E^{(-1)}$, then
except for $n\equiv 2\mod 12$ case,
the primitive test vector $f$ for $(E^{(-1)},\chi_d)$
only takes values in odd integers, and
$$
\left|\frac{P_\chi(f)}{2^{\mu(D)-1}}\right|^2
=C_{n,d}\cdot\CL(-n/d)\CL(d)\cdot T_{n,d}
$$
where
$$
C_{n,d}=1
\quad\text{and}\quad
T_{n,d}=\begin{cases}
1/4, &\text{if }d=1, \\
1,&\text{otherwise}.
\end{cases}
$$
\end{thm}

\begin{remark}
In fact, if
$n\neq 1,3$ is a positive square-free integer, then
$(n,d)$ is admissible for $E$ if and only if
$$
\begin{cases}
d\equiv n,7\mod 24&\text{and }n\equiv 1\mod 12,\text{ or}\\
d\equiv n\mod 48\text{ or }d\equiv 7\mod 24&\text{and }n\equiv 2\mod 12,\text{ or}\\
d\equiv n\mod 72\text{ or }d\equiv 11\mod 24&\text{and }n\equiv 3\mod 24,\text{ or}\\
d\equiv n\mod 144\text{ or }d\equiv n-16\mod 48&\text{and }n\equiv 6\mod 12,\text{ or}\\
d\equiv 7,11\mod 24&\text{and }n\equiv 7,11\mod 24,
\end{cases}
$$
$(n,d)$ is admissible for $E^{(-1)}$ if and only if
$$
\begin{cases}
d\equiv n,7\mod 24&\text{and }n\equiv 1,5\mod 24,\text{ or}\\
d\equiv n\mod 48\text{ or }d\equiv 7\mod 24&\text{and }n\equiv 2\mod 12,\text{ or}\\
d\equiv n\mod 36&\text{and }n\equiv 3\mod 12,\text{ or}\\
d\equiv 7\mod 12&\text{and }n\equiv 7\mod 12,\text{ or}\\
d\equiv n,n-18\mod 72&\text{and }n\equiv 9\mod 24.
\end{cases}
$$
\end{remark}

\subsection{Induction method}

The following result relates the $\CL(\pm n)$
with $\CL(\pm d)$ where $d$ is a divisor of $n$,
whose number of prime divisors is smaller than that of $n$.

\begin{prop}
\label{analytic sha ind 1}
Let $E:y^2=x(x-1)(x+3)$ and $n\neq 1,3$ be a positive square-free integer.

{\rm(i)}
For $n$ such that $E^{(-n)}$ has sign $+1$, $\CL(-n)$ is a square of an integer;
moreover, for $n\equiv 1\mod 12$, $\CL(-n)$ is even.
We have the following induction equalities for $\CL(-n)$ in $\BF_2$:
\begin{align*}
\CL(-n)&=g(n)
\qquad\text{for}\qquad n\equiv 2,6\mod 12\text{ or }n\equiv 3,11\mod 24, \\
\CL(-n)+\sum_{\substack{d\neq 1,n\\
d\mid n\\
d\equiv 11\mod 24}}\CL(n/d)\CL(-d)&=g(n)
\qquad\text{for}\qquad n\equiv 7\mod 24.
\end{align*}
Here $g(n):=\#2\Cl(\BQ(\sqrt{-n}))$.

{\rm(ii)}
For $n$ such that $E^{(n)}$ has sign $+1$, $\CL(n)$ is a square of an integer.
We have the following induction equalities for $\CL(n)$ in $\BF_2$:
\begin{align*}
\CL(n)+\sum_{\substack{d\neq 1,n\\
d\mid n\\
d\equiv 7\mod 24}}\CL(-n/d)\CL(d)&=P_0(f)
\qquad\text{for}\qquad n\equiv 2\mod 12, \\
\CL(n)&=g(n)
\qquad\text{for}\qquad n\equiv 3,7\mod 12, \\
\CL(n)+\sum_{\substack{d\neq 1,n\\
d\mid n\\
d\equiv 7\mod 24}}\CL(-n/d)\CL(d)&=g(n)
\qquad\text{for}\qquad n\equiv 5\mod 24, \\
\CL(n)+\sum_{\substack{d\neq 1,n\\
d\mid n\\
d\equiv n-18\mod 72}}\CL(-n/d)\CL(d)&=g(n)
\qquad\text{for}\qquad n\equiv 9\mod 24.
\end{align*}
Here $P_0(f)$ is the genus period associated to
$(E^{(-1)},1_{\BQ(\sqrt{-n})})$ in $n\equiv 2\mod 12$ case.
\end{prop}

\begin{proof}
First we come
back to the settings in Proposition \ref{period0}.
Assume the root number of $E$ is $+1$,
and $n$ is a positive square-free integer such that the root number of
$E^{(-n)}$ is also $+1$. Let $f$ be a primitive test vector for $(E,1_K)$.
Consider all unramified quadratic characters $\chi=\chi_d$ of $K$.
If $d$ is such that both of $E^{(n/d)}$ and $E^{(-d)}$
have root number $+1$, and such that $\chi_{d,v}=1_{K_v}$
for all $v\mid(N,D)$, then $f$ is also a primitive test vector for $(E,\chi)$
and $P_\chi(f)$ is related to $\CL(n/d)\CL(-d)$ by Proposition \ref{period0}.
If $d$ does not satisfy these conditions,
then $P_\chi(f)=0$.

Sum up all these $\chi_d$ we obtain
\begin{equation}
\label{e:genus-period-1}
\sum_{\substack{\chi_d:\Gal(H_K/K)\to\{\pm 1\}\\
f\in V(E,\chi_d)}}\frac{P_{\chi_d}(f)}{2^{\mu(D)-1}}
=P_0(f),
\end{equation}
where $P_0(f)$ is the genus period:
$$
P_0(f):=\sum_{t\in\Gal(H_K/H_0)}f(\iota(t))\in\BZ.
$$
If $f$ only takes values in odd integers, then $P_0(f)\equiv g(n)\mod 2$.

Similarly, let $H_0'$ be a subfield of $H_0$ with $[H_0:H_0']=2$ such that $H_0=H_0'(\sqrt{-1})$.
Such choice of $H_0'$ makes $\chi_1:\Gal(H_K/K)\to\{\pm 1\}$ \emph{does not}
factor through $\Gal(H_0'/K)$, and we have
\begin{equation}
\label{e:genus-period-2}
\sum_{\substack{\chi_d:\Gal(H_0'/K)\to\{\pm 1\}\\
f\in V(E,\chi_d)}}\frac{P_{\chi_d}(f)}{2^{\mu(D)-1}}
=\frac{P_0'(f)}{2},
\end{equation}
where
$$
P_0'(f):=\sum_{t\in\Gal(H_K/H_0')}f(\iota(t))\in\BZ.
$$
Since $[H_K:H_0']$ is even, we obtain that if $f$ only takes values
in odd integers (or even integers),
then $P_0'(f)\equiv 0\mod 2$ and the right hand side of \eqref{e:genus-period-2}
is an integer.

Apply the above discussions to $E:y^2=x(x-1)(x+3)$ as well as its quadratic twist
$E^{(-1)}$. By \eqref{e:genus-period-1}
(\eqref{e:genus-period-2} for $n\equiv 1\mod 24$) and Theorem \ref{period1},
induction on the number of prime factors of $n$,
it's easy to obtain the integrality and $2$-divisibility of $\CL(\pm n)$.
Once we know the integrality, the induction equalities come easily from
\eqref{e:genus-period-1} and Theorem \ref{period1}.
\end{proof}

The above induction equalities in
Proposition \ref{analytic sha ind 1} allow us to express $\CL(\pm n)$
in terms of $g(n)$.

\begin{thm}
\label{analytic sha main}
Let $E:y^2=x(x-1)(x+3)$
and $n\neq 1,3$ be a positive square-free integer.
Then we have the following equalities for $\CL(-n)$ in $\BF_2$:
\begin{align*}
\CL(-n)&=g(n) &
\text{for}\quad n&\equiv 2,6\mod 12\text{ or } n\equiv 3,11\mod 24, \\
\CL(-n)&=g(n)+\sum_{\substack{d\neq 1,n\\
d\mid n\\
d\equiv 11\mod 24}}g(n/d)g(d) &
\text{for}\quad n&\equiv 7\mod 24,
\end{align*}
and the following equalities for $\CL(n)$ in $\BF_2$:
\begin{align*}
\CL(n)&=P_0(f)+\sum_{\substack{d\neq 1,n\\
d\mid n\\
d\equiv 7\mod 24}}g(n/d)g(d) &
\text{for}\quad n&\equiv 2\mod 12, \\
\CL(n)&=g(n) &
\text{for}\quad n&\equiv 3,7\mod 12, \\
\CL(n)&=g(n)+\sum_{\substack{d\neq 1,n\\
d\mid n\\
d\equiv 7\mod 24}}g(n/d)g(d) &
\text{for}\quad n&\equiv 5\mod 24, \\
\CL(n)&=g(n)+\sum_{\substack{d\neq 1,n\\
d\mid n\\
d\equiv n-18\mod 72}}\left(g(n/d)+\sum_{\substack{d'\neq 1,n/d\\
d'\mid n/d\\
d'\equiv 11\mod 24}}g(n/dd')g(d')\right)g(d) &
\text{for}\quad n&\equiv 9\mod 24.
\end{align*}
Here $g(n):=\#2\Cl(\BQ(\sqrt{-n}))$
and $P_0(f)$ is the genus period associated to
$(E^{(-1)},1_{\BQ(\sqrt{-n})})$ in $n\equiv 2\mod 12$ case.
\end{thm}

\begin{proof}
Induction on the number of prime factors of $n$.
The formula of $\CL(n)$ for $n\equiv 5\mod 24$ is obtained as follows:
if $n\equiv 5\mod 24$ and $d\equiv 7\mod 24$,
then $n/d\equiv 11\mod 24$, so
$$
\CL(n)=g(n)+\sum_{\substack{d\neq 1,n\\
d\mid n\\
d\equiv 7\mod 24}}\CL(-n/d)\CL(d)
=g(n)+\sum_{\substack{d\neq 1,n\\
d\mid n\\
d\equiv 7\mod 24}}g(n/d)g(d).
$$
If $n\equiv 7\mod 24$ and $d\equiv 11\mod 24$,
then $n/d\equiv 5\mod 24$, so
\begin{align*}
\CL(-n)&=g(n)+\sum_{\substack{d\neq 1,n\\
d\mid n\\
d\equiv 11\mod 24}}\CL(n/d)\CL(-d)
=g(n)+\sum_{\substack{d\neq 1,n\\
d\mid n\\
d\equiv 11\mod 24}}
\left(g(n/d)+\sum_{\substack{d'\neq 1,n/d\\
d'\mid n/d\\
d'\equiv 7\mod 24}}g(n/dd')g(d')\right)g(d) \\
&=g(n)+\sum_{\substack{d\neq 1,n\\
d\mid n\\
d\equiv 11\mod 24}}
g(n/d)g(d)+\sum_{\substack{dd'd''=n\\
d\neq 1,d'\neq 1,d''\neq 1\\
(d,d',d'')\equiv(11,7,11)\mod 24}}g(d'')g(d')g(d).
\end{align*}
The last term is zero because there is a symmetry $d\leftrightarrow d''$.
Hence the formula of $\CL(-n)$ for $n\equiv 7\mod 24$ follows.
The others are similar.
\end{proof}

\subsection{$L$-values and ternary quadratic forms}

Integral definite ternary quadratic form is a classical object in number theory.
Our study on toric periods gives the following result.

\begin{thm}
\label{ternary 7}
For an integral positive definite ternary quadratic form $Q$, define
$$
r_Q(n):=\# \left\{ (x, y, z)\in \BZ^3\ \Big|\ Q(x, y, 2z)=n\right\}
-\frac{1}{2}\# \left\{ (x, y, z)\in \BZ^3\ \Big|\ Q(x, y, z)=n\right\}.
$$
Let $E:y^2=x(x-1)(x+3)$ and $n\equiv 7\mod 24$ be a positive square-free integer.
Then the analytic Sha of $E^{(\pm n)}$ satisfies
\begin{align*}
\CL(-n)&=4^{-\mu(n)}r_Q(n)^2
\quad\text{with}\quad
Q(x,y,z)=x^2+3y^2+36z^2, \\
\CL(n)&=4^{-\mu(n)}r_Q(n)^2
\quad\text{with}\quad
Q(x,y,z)=x^2+3y^2+12z^2.
\end{align*}
Similarly, for $n\equiv 3\mod 24$ positive square-free, we have
\begin{align*}
\CL(-n)&=4^{-\mu(n/3)}r_Q(n/3)^2
\quad\text{with}\quad
Q(x,y,z)=6x^2+y^2+2z^2, \\
\CL(n)&=4^{-\mu(n/3)}r_Q(n/3)^2
\quad\text{with}\quad
Q(x,y,z)=x^2+3y^2+4z^2.
\end{align*}
\end{thm}

\begin{proof}
Here we only prove the $n\equiv 7\mod 24$ case; the $n\equiv 3\mod 24$ case is similar.
Let $\phi$ be the newform associated to $E$.
Let $f=\sum_{n=1}^\infty b_n q^n$ be an eigenform
in $S_{3/2}(N,\chi_{d_1})$ where $N$ is a multiple of the level of $\phi$,
and $\chi_{d_1}$ is the quadratic character of $\BQ(\sqrt{d_1})/\BQ$,
such that $\phi^{(d_2)}$ is the Shimura lifting of $f$.
Let $\fX$ be a $\Sigma$-equivalence class with root number $+1$,
such that
there exists a square-free integer $n\in d_1d_2\fX$
satisfying
$$
C:=\frac{L(\phi^{(-nd_1d_2)},1)\sqrt{n}\chi_{d_1}(n)}{b_n^2}\neq 0.
$$
Then by Waldspurger's theorem \cite{Wald},
for any square-free integer $n'\in d_1d_2\fX$,
$b_{n'}\neq 0$ if and only if $L(\phi^{(-n'd_1d_2)},1)\neq 0$,
and the constant $C$
only depends on the $\Sigma$-equivalence class $\fX$:
$$
C=\frac{L(\phi^{(-n'd_1d_2)},1)\sqrt{n'}\chi_{d_1}(n')}{b_{n'}^2}.
$$

The theorem for $\CL(-n)$ is obtained by the existence of
$f=\sum_{n=1}^\infty b_nq^n$ with $N=576$, $d_1=d_2=3$,
satisfying the above conditions, and
such that $b_n=r_Q(n)$ for all $n\equiv 7\mod 24$ square-free.

For the $\CL(n)$,
taking $Q(x,y,z)=x^2+3y^2+12z^2$,
then there is $f=\sum_{n=1}^\infty b_nq^n$ with $N=192$, $d_1=1$, $d_2=-1$,
satisfying the above conditions, and such that
$b_n=r_Q(n)$ for all $n\equiv 7\mod 24$ square-free,
hence by the same method we obtain the desired result.
\end{proof}

\begin{remark}
By looking at solutions with $xyz=0$, one can easily see that
for $Q(x,y,z)=x^2+3y^2+36z^2$ or
$Q(x,y,z)=x^2+3y^2+12z^2$,
for any prime $\ell\equiv 7\mod 24$, the $4^{-1}r_Q(\ell)$ is an odd integer.
\end{remark}

\section{$2$-Selmer groups and positive density}

Let $E/\BQ$ be an elliptic curve. Then the $2$-Selmer group
of $E$ is
$$
\Sel_2(E/\BQ):=\Ker \left(H^1(\BQ, E[2])\lra
\prod_v H^1(\BQ_v, E[2])\right).
$$
It is a finite dimensional $\BF_2$-vector space, and fits into the short exact sequence
$$
0\to E(\BQ)/2E(\BQ)\to\Sel_2(E/\BQ)\to\Sha(E/\BQ)[2]\to 0.
$$
In particular, if $\Sel_2(E/\BQ)$ modulo
torsion is trivial,
then $E$ is of Mordell-Weil rank $0$ and $\Sha(E/\BQ)[2]=0$.
The Propositions \ref{descent main} and \ref{descent main 3} are some examples
for tiling number elliptic curves
$\pm ny^2=x(x-1)(x+3)$
such that this situation occurs.
On the other hand, Proposition \ref{descent non-trivial example}
is an example that in some cases the $2$-Selmer group is always non-trivial,
hence under the assumption that $E$ is of Mordell-Weil rank $0$,
the $\Sha(E/\BQ)[2]$ is always non-trivial.

A useful property of $2$-Selmer group is
the $2$-parity conjecture, first proved by Monsky \cite{Mon96},
which asserts that $\dim_{\BF_2}\Sel_2(E/\BQ)\equiv
\displaystyle\operatorname*{ord}_{s=1}L(E,s)\mod 2$.

\subsection{$2$-descent for tiling number curves}
\label{s:descent}

The $2$-descent is a commonly-used method to study the structure of the Mordell-Weil
group of an elliptic curve. We briefly recall it as follows, see \cite{GTM106},
Chapter X.
The advantage of the $2$-descent method is that the
elements of the $2$-Selmer group
$\Sel_2(E/\BQ)$ correspond to the curves which are $2$-covers of $E$, and which
can be computed explicitly in terms of linear algebra.
In the following we do such computations for the tiling number elliptic curves.
It can also be computed in terms of graph theoretic language \cite{Feng},
for the tiling number elliptic curves
see the work of Qiuyue Liu, Jing Yang, Keqin Feng \cite{LYF},
and earlier works by Yoshida \cite{Yoshida}
and Goto \cite{Goto1} \cite{Goto2}.

\subsubsection{}

Let $m$ be a square-free integer and $E^{(m)}$ the elliptic curve over $\BQ$ defined by the Weierstrass equation $y^2=x(x-m)(x+3m)$.
Let $\Sigma$ be the set of primes containing $2,\infty$
and bad places for $E$, namely, $\Sigma=\{2,3,\infty\}$.
Let $S$ be the set of primes dividing $6m\infty$,
and $\Sigma_1:=S\setminus\Sigma$.
Let $\BQ(S, 2)$ be the subgroup
of $\BQ^\times/(\BQ^\times)^2$ supported on $S$, or equivalently, the set of square-free integers with prime factors in $S$. Then elements in $\Sel_2(E^{(m)}/\BQ)$ can be  realized as curves $C_\Lambda/\BQ$ with $C_\Lambda(\BA_\BQ)\neq \emptyset$.  Here for $\Lambda=(b_1, b_2)\in \BQ(S, 2)^{\oplus 2}$, define $C_\Lambda$ to be
$$ C_\Lambda: \qquad \begin{cases}b_1z_1^2-b_2z_2^2=mt^2\\
 b_1z_1^2-b_1b_2z_3^2=-3mt^2	
 \end{cases}$$
Note that $C_\Lambda(\BA_\BQ)\neq \emptyset$ if and only if $C_\Lambda(\BQ_v)\neq \emptyset$ for all $v\in S$.
The four $2$-torsion points $O$, $(0, 0)$, $(m, 0)$, and $(-3m, 0)$ correspond to
$(b_1, b_2)=(1, 1)$, $(-3, -m)$, $(m, 1)$, and $(-3m, -m)$.

For any $p\in \Sigma_1$,  $C_\Lambda(\BQ_p)\neq \emptyset$ if and only if
\begin{equation*}
\begin{cases}
\lrd{b_1}{p}=\lrd{b_2}{p}=1, &\text{if } p\nmid b_1b_2, \\
\lrd{m/b_1}{p}=\lrd{b_2}{p}=1, &\text{if }p\mid b_1,p\nmid b_2, \\
\lrd{-3b_1}{p}=\lrd{-m/b_2}{p}=1, &\text{if }p\nmid b_1,p\mid b_2, \\
\lrd{-3m/b_1}{p}=\lrd{-m/b_2}{p}=1, &\text{if }p\mid b_1, p\mid b_2.
\end{cases}
\end{equation*}
For $p\in \Sigma$, it turns out that $C_\Lambda(\BQ_p)\neq\emptyset$
if and only if $\loc_p(\Lambda)$ is contained in certain subspace
$W_p$ of $(\BQ_p^\times/(\BQ_p^\times)^2)^{\oplus 2}$
which only depends on the image of $m$ in $\BQ_p^\times/(\BQ_p^\times)^2$.
For example,
if $3\nmid m$, then $C_\Lambda(\BQ_3)\neq \emptyset$ if and only if
$3\nmid b_2$ and
\begin{equation*}
\begin{cases}
\lrd{b_2}{3}=1,&\text{if }3\nmid b_1,\\
\lrd{-b_2m}{3}=1, &\text{if }3\mid b_1,
\end{cases}
\end{equation*}
and if $3\mid m$, then $C_\Lambda(\BQ_3)\neq \emptyset$ if and only if
\begin{equation*}
\begin{cases}
\lrd{b_1}{3}=\lrd{b_2}{3}=1, &\text{if } 3\nmid b_1b_2, \\
\lrd{m/b_1}{3}=\lrd{b_2}{3}=1, &\text{if }3\mid b_1,3\nmid b_2, \\
\lrd{-m/3b_1}{3}=\lrd{-m/b_2}{3}=1, &\text{if }3\nmid b_1,3\mid b_2, \\
\lrd{-b_1/3}{3}=\lrd{-m/b_2}{3}=1, &\text{if }3\mid b_1, 3\mid b_2.
\end{cases}
\end{equation*}
It's easy to see that $C_\Lambda(\BR)\neq \emptyset$ if and and only if
$$\begin{cases} b_1b_2>0, \quad &\text{if } m>0,\\
b_2>0, &\text{if } m<0.	
\end{cases}$$
If $2\nmid m$, then it is easy to see that
if $C_\Lambda(\BQ_2)\neq \emptyset$ then $b_1$ must be odd, and one may compute the precise condition for $C_\Lambda(\BQ_2)\neq \emptyset$.
The followings are some examples.
\begin{itemize}
\item
In the case $m\equiv 3\mod 4$,   $C_\Lambda(\BQ_2)\neq \emptyset$ if and only if
$2\nmid b_1$ and $b_2\equiv 1\mod 4$.
\item
In the case $m\equiv 5\mod 8$,   $C_\Lambda(\BQ_2)\neq \emptyset$ if and only if
$b_1\equiv 1\mod 4$ and $2\nmid b_2$.
\item
In the case $m\equiv 1\mod 8$, $C_\Lambda(\BQ_2)\neq \emptyset$ if and only if
$(b_1\mod 8, b_2\mod 8)=(1,1)$, $(1,5)$, $(3,2)$, $(5,3)$, $(5,7)$ or $(7,6)$.
Equivalently, $2\nmid b_1$
and $\left(
\lrdd{-1}{b_1},
\lrdd{2}{b_1},
\lrdd{-1}{b_2/2^{\ord_2(b_2)}},
\lrdd{2}{b_2/2^{\ord_2(b_2)}},
\ord_2(b_2)\right)^\RT$
lies in the right kernel of
$\begin{pmatrix}
1 & 1 & 1 & 0 & 0 \\
1 & 0 & 0 & 0 & 1
\end{pmatrix}$.
\end{itemize}

\subsubsection{}
\label{s:matrices}

We introduce some matrices with entries in $\BF_2$.
If $n=\ell_1\cdots \ell_k$ is an odd positive integer,
then define the matrix $A=A(n)$ associated to $n$
to be $A=(a_{ij})\in M_{k\times k}(\BF_2)$,
where $a_{ij}=\lrdd{\ell_j}{\ell_i}
:=\frac12\left(1-\left(\frac{\ell_j}{\ell_i}\right)\right)$
(the additive Legendre symbol) if $i\neq j$ and $\sum_j a_{ij}=0$ for all $i$,
namely $a_{ii}=\left[\frac{n/\ell_i}{\ell_i}\right]$.
For an integer $d$ prime to $n$, define the column vector $z_d=z_d(n)$ associated to $n$
to be $z_d:=\left(\lrdd{d}{\ell_1},\cdots,\lrdd{d}{\ell_k}\right)^\RT$.
Finally, denote $D_d:=\diag(z_d)$.
It's easy to see that $A+A^{\mathrm T}=z_{-1}z_{-1}^\RT+D_{-1}$.

\begin{lem}
\label{A+D-1}
Let $A$ be the matrix associated to $n$ with $n\equiv 3\mod 4$.
Then $\rank(A+D_{-1})=\rank(A)+1$.
\end{lem}

\begin{proof}
Since $n$ has odd number of prime factors $\equiv 3\mod 4$,
it's easy to see that we also have $\sum_ia_{ij}=0$ for all $j$.
Therefore
$$
\rank(A+D_{-1})=\rank\begin{pmatrix}
A+D_{-1} \\ z_{-1}^{\mathrm T}
\end{pmatrix}
=\rank\begin{pmatrix}
A+D_{-1}+z_{-1}z_{-1}^{\mathrm T} \\ z_{-1}^{\mathrm T}
\end{pmatrix}
=\rank\begin{pmatrix}
A^{\mathrm T} \\ z_{-1}^{\mathrm T}
\end{pmatrix}
=\rank(A)+1,
$$
where the last equality holds since
the elements of $z_{-1}^{\mathrm T}$ sum to $1$
while the elements of any row of $A^{\mathrm T}$ sum to $0$.
\end{proof}

\begin{remark}
When $n\equiv 3\mod 4$,
it's known that $\rank_4\Cl(\BQ(\sqrt{-n}))
=\corank(A)-1=\corank(A+D_{-1})$.
Here for a matrix $M$, the $\corank(M)$ is the dimension
of the right kernel of $M$.
In particular, $g(n)\equiv\det(A+D_{-1})\mod 2$
where $g(n):=\#2\Cl(\BQ(\sqrt{-n}))$.
\end{remark}

\subsubsection{}

In the following we consider $E^{(m)}:y^2=x(x-m)(x+3m)$ with
$m$ square-free and varies in a fixed $\Sigma$-equivalence class.
Write $m=q n$ where $q\mid 6$
and $n=\ell_1\cdots \ell_k$ is coprime to $6$.
Let $\Lambda=(b_1,b_2)\in\BQ(S,2)^{\oplus 2}$.
For $t=1,2$,
write $b_t=c_t \prod_i \ell_i^{x_{t,i}}$ with $x_{t, i}\in \BF_2$ and $c_t\mid 6$,
furthermore, write $c_t=\prod_{p\in\Sigma}p^{y_t^{(p)}}$
with $y_t^{(p)}\in\BF_2$,
here the $\infty\in \Sigma$ is regarded as $-1$.
It's clear that $z_{c_t}=\sum_{p\in\Sigma}z_py_t^{(p)}$.
The condition that
$C_\Lambda(\BQ_v)\neq\emptyset$ for all $v\in \Sigma_1$
is equivalent to the equation
$$\matrixx{A+D_q}{D_{-3}}{}{A+D_{-q}} \begin{pmatrix} x_1\\ x_2\end{pmatrix}= \begin{pmatrix} z_{c_1}\\ z_{c_2}\end{pmatrix}.$$
For $v\in \Sigma$, the condition that $C_\Lambda(\BQ_v)=0$
is equivalent to $\loc_v(\Lambda)\in W_p$ for certain subspace $W_p$
of $(\BQ_p^\times/(\BQ_p^\times)^2)^{\oplus 2}$ which only depends on the $\Sigma$-equivalence class
of $m$, and which can be expressed as a linear equation in $\BF_2$
involving $x_t$ and $y_t^{(p)}$.
For example, if $m\equiv 1\mod 8$, then the condition that
$C_\Lambda(\BQ_2)\neq\emptyset$ is equivalent to $y_1^{(2)}=0$ and
$$
\begin{pmatrix}
z_{-2}^\RT & z_{-1}^\RT & 1 & 0 & 1 & 0 & 1 \\
z_{-1}^\RT & 0 & 1 & 1 & 0 & 1 & 0
\end{pmatrix}
\begin{pmatrix}
x_1 \\ x_2 \\ y_1^{(-1)} \\ y_1^{(3)} \\ y_2^{(-1)} \\ y_2^{(2)} \\ y_2^{(3)}
\end{pmatrix}
=0.
\qquad
$$
Therefore, when $m$ varies in a fixed $\Sigma$-equivalence class,
whether an element $\Lambda=(b_1,b_2)$
is contained in $\Sel_2(E^{(m)}/\BQ)$ can be expressed in term
of a linear equation in $\BF_2$ involving $x_t$ and $y_t^{(p)}$,
whose form only depends on the $\Sigma$-equivalence class.
In the following we will see some examples.

The following is the result for $E^{(\pm n)}$, $n\equiv 7\mod 24$ case.

\begin{prop}\label{descent main}
{\rm(i)}
Let $n\equiv 7\mod 12$ be a positive square-free integer.
Then
$$
\dim_{\BF_2}\big(\Sel_2(E^{(n)}/\BQ)\big/ E^{(n)}[2]\big)
=2k-\rank\begin{pmatrix}
A & D_{-3} & z_3 \\
& A+D_{-1} &
\end{pmatrix}.
$$
In particular,
\begin{itemize}
\item
$\BQ(\sqrt{-n})$ has no ideal classes of order $4$ if and only if
$\Sel_2(E^{(n)}/\BQ)\big/ E^{(n)}[2]=0$.
\item
If all the prime factors of $n$
are $\equiv 1\mod 3$, then
$$
\dim_{\BF_2}\big(\Sel_2(E^{(n)}/\BQ)\big/ E^{(n)}[2]\big)
=2\cdot\rank_4\Cl(\BQ(\sqrt{-n})).
$$
\end{itemize}

{\rm(ii)}
Let $n\equiv 7\mod 24$ be a positive square-free integer. Then
$$
\dim_{\BF_2}\big(\Sel_2(E^{(-n)}/\BQ)\big/ E^{(-n)}[2]\big)
=2k+1-\rank\begin{pmatrix}
A+D_{-1} & D_{-3} & \\
& A & z_2 \\
& z_{-3}^{\mathrm T} & 1 \\
z_{-2}^{\mathrm T} & z_{-1}^{\mathrm T} &
\end{pmatrix}.
$$
In particular,
\begin{itemize}
\item
$\det\begin{pmatrix}
z_{-1}z_{-2}^{\mathrm T}+z_2z_{-1}^{\mathrm T} & A^{\mathrm T}+D_{-1} \\
A+D_{-1} & D_{-3} \\
\end{pmatrix}$ is odd
if and only if
$\Sel_2(E^{(-n)}/\BQ)\big/ E^{(-n)}[2]=0$.
\item
If all the prime factors of $n$
are $\equiv 1\mod 3$, then
$$
\dim_{\BF_2}\big(\Sel_2(E^{(-n)}/\BQ)\big/ E^{(-n)}[2]\big)
=2\cdot\rank_4\Cl(\BQ(\sqrt{-n})).
$$
\end{itemize}
\end{prop}
\begin{proof}
(i)
First we prove that
when $n\equiv 7\mod 12$ is positive square-free,
the $\Sel_2(E^{(n)}/\BQ)$ is isomorphic to the right kernel of the matrix
$$
B:=\begin{pmatrix}
A & D_{-3} & z_{-1} & z_3 \\
& A+D_{-1} & z_{-1} & \\
& z_{-1}^{\mathrm T} & 1 & \\
& z_{-3}^{\mathrm T} & 1 & 1 \\
\end{pmatrix}.
$$
For all other cases in this section, there are similar results.
We explain the linear algebra details for this case, and omit them for all other cases,
since the method is very similar:
the condition that $C_\Lambda(\BQ_v)\neq\emptyset$ for $v=\infty,2,3$
translates as
\begin{itemize}
\item
$v=\infty$: $y_1^{(-1)}=y_2^{(-1)}$,
\item
$v=2$: $y_1^{(2)}=y_2^{(2)}=0$ and $y_2^{(-1)}+y_2^{(3)}+z_{-1}^{\mathrm T}x_2=0$,
\item
$v=3$: $y_2^{(3)}=0$ and $y_1^{(3)}+y_2^{(-1)}+y_2^{(2)}+z_{-3}^{\mathrm T}x_2=0$,
\end{itemize}
therefore it's easy to see that
$(b_1,b_2)\in\Sel_2(E^{(n)}/\BQ)$ if and only if
$y_1^{(-1)}=y_2^{(-1)}$, $y_1^{(2)}=y_2^{(2)}=y_2^{(3)}=0$ and
$\left(\begin{smallmatrix}
x_1 \\ x_2 \\ y_1^{(-1)} \\ y_1^{(3)}
\end{smallmatrix}\right)$
lies in the right kernel of $B$.

Notice that in the $(2k+2)\times(2k+2)$ matrix $B$,
the rows $1$ to $k$ (resp.~rows $k+1$ to $2k$,
resp.~columns $k+1$ to $2k$ together with $2k+2$) sum to the row $2k+2$
(resp.~the row $2k+1$, resp.~the column $2k+1$),
so the desired result follows easily.

Assume that $\BQ(\sqrt{-n})$ has no order $4$ ideal classes,
then $\rank(A)=k-1$ and $A+D_{-1}$ is invertible,
hence
$$
\rank\begin{pmatrix}
A & D_{-3} & z_3 \\
& A+D_{-1} &
\end{pmatrix}
=\rank\begin{pmatrix}
A & & z_3 \\
& A+D_{-1} &
\end{pmatrix}
=\rank(A)+\rank(A+D_{-1})+1=2k,
$$
so $\Sel_2(E^{(n)}/\BQ)\big/ E^{(n)}[2]=0$.
Conversely,
assume that $\BQ(\sqrt{-n})$ has an order $4$ ideal class,
then $\rank(A+D_{-1})\leq k-1$,
hence
$$
\rank\begin{pmatrix}
A & D_{-3} & z_3 \\
& A+D_{-1} &
\end{pmatrix}
\leq k+\rank(A+D_{-1})\leq 2k-1,
$$
so $\Sel_2(E^{(n)}/\BQ)\big/ E^{(n)}[2]\neq 0$.
Finally, if all the prime factors of $n$
are $\equiv 1\mod 3$, then $D_{-3}=0$, it's easy to see that
the desired result holds.

(ii)
In this case
the $\Sel_2(E^{(-n)}/\BQ)$ is isomorphic to the right kernel of the matrix
$$
B:=\begin{pmatrix}
A+D_{-1} & D_{-3} & z_{-1} & z_3 & \\
& A & & & z_2 \\
& z_{-3}^{\mathrm T} & & & 1 \\
z_{-2}^{\mathrm T} & z_{-1}^{\mathrm T} & 1 & & \\
z_{-1}^{\mathrm T} & & 1 & 1 & 1
\end{pmatrix}.
$$
Notice that in the $(2k+3)\times(2k+3)$ matrix $B$,
the columns $1$ to $k$ (resp.~columns $1$ to $2k$,
resp.~rows $1$ to $k$ together with $2k+1$) sum to the column $2k+1$
(resp.~the column $2k+2$, resp.~the row $2k+3$),
so the desired result follows easily.

Since the sum of elements of $z_{-1}$ (resp.~$z_2$) is $1$
(resp.~$0$), and that $A+A^{\mathrm T}=z_{-1}z_{-1}^{\mathrm T}+D_{-1}$, we have
\begin{align*}
\rank\begin{pmatrix}
A+D_{-1} & D_{-3} & \\
& A & z_2 \\
& z_{-3}^{\mathrm T} & 1 \\
z_{-2}^{\mathrm T} & z_{-1}^{\mathrm T} &
\end{pmatrix}
&=\rank\begin{pmatrix}
A+D_{-1} & D_{-3} & \\
z_{-1}z_{-2}^{\mathrm T} & A^{\mathrm T}+D_{-1} & z_2 \\
z_{-1}^{\mathrm T} & & 1 \\
\end{pmatrix} \\
&=\rank\begin{pmatrix}
A+D_{-1} & D_{-3} \\
z_{-1}z_{-2}^{\mathrm T}+z_2z_{-1}^{\mathrm T} & A^{\mathrm T}+D_{-1} \\
\end{pmatrix}+1,
\end{align*}
which gives the result on the condition of
$\Sel_2(E^{(-n)}/\BQ)\big/ E^{(-n)}[2]=0$.
Finally, if all the prime factors of $n$
are $\equiv 1\mod 3$, then $D_{-3}=0$ and $z_{-3}^{\mathrm T}=0$,
it's easy to see that the desired result holds.
\end{proof}

The following example illustrates that in $n\equiv 1\mod 12$ case the $2$-Selmer group
of $E^{(-n)}$
modulo torsion
is always non-trivial, coincides with the $2$-divisibility of analatic Sha
$\CL(-n)$
in Proposition \ref{analytic sha ind 1}.

\begin{prop}
\label{descent non-trivial example}
Let $n\neq 1$, $n\equiv 1\mod 12$ be a positive square-free integer.
Then
$$
\dim_{\BF_2}\big(\Sel_2(E^{(-n)}/\BQ)\big/ E^{(-n)}[2]\big)
=2k-\rank\begin{pmatrix}
A+D_{-1} & D_{-3} \\
& A
\end{pmatrix},
$$
which is always a positive integer
(thus $\geq 2$ by parity).
\end{prop}
\begin{proof}
In this case
the $\Sel_2(E^{(-n)}/\BQ)$ is isomorphic to the right kernel of the matrix
$$
B:=\begin{pmatrix}
A+D_{-1} & D_{-3} & z_{-1} & z_3 \\
& A & & \\
& z_{-3}^{\mathrm T} & & \\
& z_{-1}^{\mathrm T} & & \\
\end{pmatrix}.
$$
Note that since $n\equiv 1\mod 4$, the sum of rows of $A$
equals $z_{-1}^{\mathrm T}$,
hence in the $(2k+2)\times(2k+2)$ matrix $B$,
the columns $1$ to $k$ (resp.~columns $k+1$ to $2k$,
resp.~rows $1$ to $k$, resp.~rows $k+1$ to $2k$) sum to the column $2k+1$
(resp.~the column $2k+2$, resp.~the row $2k+1$, resp.~the row $2k+2$),
so the desired result follows easily
(note that $\rank(A)\leq k-1$).
\end{proof}

The following is the result for $E^{(\pm 3n)}$,
$3n\equiv 3\mod 24$ case.

\begin{prop}
\label{descent main 3}
Let $q\in\{\pm 3\}$ and
$n\neq 1$, $n\equiv 1\mod 4$ be a positive square-free integer coprime to $6$.
If $q=-3$, further require that $n\equiv 1\mod 8$.
Then
$$
\dim_{\BF_2}\big(\Sel_2(E^{(qn)}/\BQ)\big/ E^{(qn)}[2]\big)
=2k+1-\left\{\begin{array}{ll}
\rank\begin{pmatrix}
A+D_3 & D_{-3} & z_{-1} \\
& A+D_{-3} & z_{-1} \\
& z_{-1}^{\mathrm T} & 1 \\
\end{pmatrix},&\text{if }q=+3, \\[1.6em]
\rank\begin{pmatrix}
A+D_3 & & \\
D_{-3} & A+D_{-3} & z_{-1} \\
& z_{-1}^{\mathrm T} & 1 \\
\end{pmatrix},&\text{if }q=-3.
\end{array}\right.
$$
In particular,
\begin{itemize}
\item
$\BQ(\sqrt{-3n})$ has no ideal classes of order $4$ if and only if
$\Sel_2(E^{(qn)}/\BQ)\big/ E^{(qn)}[2]=0$.
\item
If all the prime factors of $n$
are $\equiv 1\mod 3$, then
$$
\dim_{\BF_2}\big(\Sel_2(E^{(qn)}/\BQ)\big/ E^{(qn)}[2]\big)
=2\cdot\rank_4\Cl(\BQ(\sqrt{-3n})).
$$
\end{itemize}
\end{prop}

\begin{proof}
The $\Sel_2(E^{(3n)}/\BQ)$ is isomorphic to the right kernel of the matrix
$$
B_1:=\begin{pmatrix}
A+D_3 & D_{-3} & z_{-1} & z_3 & \\
& A+D_{-3} & z_{-1} & & z_3 \\
& z_{-1}^{\mathrm T} & 1 & & 1 \\
z_{-3}^{\mathrm T} & & 1 & \left[\frac{n}{3}\right] & 1+\left[\frac{n}{3}\right] \\[0.2em]
& z_{-3}^{\mathrm T} & 1 & & 1+\left[\frac{n}{3}\right]
\end{pmatrix}.
$$
Notice that in the $(2k+3)\times(2k+3)$ matrix $B_1$,
the columns $1$ to $k$ (resp.~columns $1$ to $2k+1$,
resp.~rows $1$ to $2k+1$, resp.~rows $k+1$ to $2k+1$) sum to the column $2k+2$
(resp.~the column $2k+3$, resp.~the row $2k+2$, resp.~the row $2k+3$),
so the desired result follows easily.

The $\Sel_2(E^{(-3n)}/\BQ)$ is isomorphic to the right kernel of the matrix
$$
B_2:=\begin{pmatrix}
A+D_{-3} & D_{-3} & z_{-1} & z_3 & \\
& A+D_3 & & & z_3 \\
z_{-1}^{\mathrm T} & & 1 & 1 & \\
z_{-3}^{\mathrm T} & & 1 & 1+\left[\frac{n}{3}\right] & \left[\frac{n}{3}\right] \\[0.2em]
& z_{-3}^{\mathrm T} & & & \left[\frac{n}{3}\right]
\end{pmatrix}.
$$
Notice that in the $(2k+3)\times(2k+3)$ matrix $B_2$,
the columns $1$ to $k$ together with column $2k+1$ (resp.~columns $1$ to $2k$,
resp.~rows $1$ to $2k+1$, resp.~rows $k+1$ to $2k$) sum to the column $2k+2$
(resp.~the column $2k+3$, resp.~the row $2k+2$, resp.~the row $2k+3$),
so the desired result follows easily.

The R\'edei matrix of $\BQ(\sqrt{-3n})$ is
$A(3n)=\begin{pmatrix}
A+D_3 & z_3 \\
z_{-3}^{\mathrm T} & \left[\frac{n}{3}\right]
\end{pmatrix}$, similar to Lemma \ref{A+D-1}, $\rank(A(3n))+1
=\rank\begin{pmatrix}
A+D_{-3} & z_3 \\
z_{-3}^{\mathrm T} & 1+\left[\frac{n}{3}\right]
\end{pmatrix}$. By adding the columns $1$ to $k$ (resp.~rows $1$ to $k$)
to the column $k+1$ (resp.~row $k+1$), the $A(3n)$
and $\begin{pmatrix}
A+D_{-3} & z_3 \\
z_{-3}^{\mathrm T} & 1+\left[\frac{n}{3}\right]
\end{pmatrix}$ are equivalent to
$\begin{pmatrix}
A+D_3 & 0 \\
0 & 0
\end{pmatrix}$ and $\begin{pmatrix}
A+D_{-3} & z_{-1} \\
z_{-1}^{\mathrm T} & 1
\end{pmatrix}$, respectively.
Therefore $\BQ(\sqrt{-3n})$ has no ideal classes of order $4$ if and only if
$\rank(A+D_3)=k$, and it's easy to see that this is equivalent to
$\Sel_2(E^{(qn)}/\BQ)\big/ E^{(qn)}[2]=0$.
Finally, if all the prime factors of $n$
are $\equiv 1\mod 3$, then $D_{-3}=0$ and $D_3=D_{-1}$, it's easy to see that
the desired result holds.
\end{proof}

\begin{remark}
In fact, by the same method, for each $\Sigma$-equivalence class
$\fX$
one can produce a matrix whose entries are $A$, $D_d$ and $z_d$
such that its corank is equal to the $2$-Selmer of $E^{(n)}$,
$n\in\fX$.
In particular, if $n$ is positive square-free,
then for $n\equiv 3,7\mod 12$,
the $\Sel_2(E^{(n)}/\BQ)\big/ E^{(n)}[2]=0$ if and only if $\BQ(\sqrt{-n})$
has no ideal classes of order $4$;
for $n\equiv 2,6\mod 12$ or $n\equiv 3,11\mod 24$,
the $\Sel_2(E^{(-n)}/\BQ)\big/ E^{(-n)}[2]=0$ if and only if $\BQ(\sqrt{-n})$
has no ideal classes of order $4$.
This coincides with the analytic Sha result in Theorem \ref{analytic sha main}.
\end{remark}

\subsection{Limit density and joint distribution}
\label{s:positive density}

\subsubsection{}

First we briefly recall the limit density result by
the work of Gerth \cite{Gerth}.

In this section, the matrices $A$, $z_d$ and $D_d$ have slightly
different, but strongly correlated meaning than that of \S\ref{s:matrices}.
Let $\Sigma$ be a finite set consisting of $-1$ and some primes of $\BQ$.
For each $k\geq 1$, define
$$
\Omega_{k,\Sigma}:=\left\{\big((a_{ij}\in\BF_2)_{1\leq i<j\leq k},
(z_p=(z_i^{(p)})_{1\leq i\leq k}\in\BF_2^k)_{p\in\Sigma}\big)\right\},
$$
endowed with normalized counting measure.
For each $\omega\in\Omega_{k,\Sigma}$, we define $k\times k$
matrices $A=A(\omega)$ and $D_p=D_p(\omega)$ by
\begin{itemize}
\item
$z_p=z_p(\omega):=(z_1^{(p)},\cdots,z_k^{(p)})^{\mathrm T}$,
$D_p=D_p(\omega):=\diag(z_p)$,
\item
$a_{ji}:=z_i^{(-1)}z_j^{(-1)}+a_{ij}$ for $1\leq i<j\leq k$,
$a_{ii}:=\sum_{j\neq i}a_{ij}$,
$A=A(\omega):=(a_{ij})$.
\end{itemize}
For $d\in\BQ(\Sigma,2)$, we define $z_d$ and $D_d$
as $\BF_2$-linear combinations of $z_p$ and $D_p$.

The set $\Omega_{k,\Sigma}$ is endowed with a natural $S_k$-action,
and if $n=\ell_1\cdots\ell_k$ is a positive square-free integer coprime to $\Sigma$
with $k$ prime factors,
then it produces an element $\omega(n)$ of $\Omega_{k,\Sigma}$
well-defined up to $S_k$-action,
which is $a_{ij}=\lrdd{\ell_j}{\ell_i}$ for $i\neq j$,
$a_{ii}=\left[\frac{n/\ell_i}{\ell_i}\right]$,
and $z_i^{(p)}=\lrdd{p}{\ell_i}$.
The matrices $A$, $z_d$ and $D_d$ associated to $\omega(n)$
have the same meaning as that of \S\ref{s:matrices}.

\begin{thm}[Gerth \cite{Gerth}]
Let $\Omega_0$ be a subset of $\Omega_{k,\Sigma}$
stable by $S_k$-action. Then
$$
\lim_{N\to\infty}
\BP\left(\omega(n)\in\Omega_0~\Big|~n<N
\text{ positive square-free coprime to }\Sigma
\text{ with }k\text{ prime factors}\right)=\frac{\#\Omega_0}{\#\Omega_{k,\Sigma}}.
$$
\end{thm}

Taking $\Sigma=\{-1\}$,
Gerth \cite{Gerth}
considered the following probability
$$
\delta_{k,m}:=\BP\left(\corank(A+D_{-1})=m\right)
=\BP\left(\corank(A)=m+1\right)
$$
among the subset of $\Omega_{k,\Sigma}$
such that the sum of elements of $z_{-1}$ equals $1$,
and concluded that $(\delta_{k,m})_{m=0}^\infty$ is connected by a Markov chain
as $k$ varies,
which yields the limit
$$
\delta_{\infty,m}:=\lim_{k\to\infty}\delta_{k,m}
=2^{-m^2}\prod_{i=1}^m(1-2^{-i})^{-2}\prod_{i=1}^\infty(1-2^{-i}),
$$
for example,
$\delta_{\infty,0}=\prod_{i=1}^\infty(1-2^{-i})
=0.288788\cdots$, $\delta_{\infty,1}=2\delta_{\infty,0}$
and $\delta_{\infty,2}=\frac{4}{9}\delta_{\infty,0}$.
The above theorem implies that the $\delta_{k,m}$
satisfies
$$
\delta_{k,m}=\operatorname{Prob}\big(\rank_4\Cl(\BQ(\sqrt{-n}))=m\mid
n\equiv 3\mod 4\text{ square-free positive with }k\text{ prime factors}\big),
$$
hence the $\delta_{\infty,m}$ is equal to the limit density:
$$
\delta_{\infty,m}=\limProb\big(\rank_4\Cl(\BQ(\sqrt{-n}))=m\mid
n\equiv 3\mod 4\text{ square-free positive}\big).
$$

\begin{remark}
Via analytic tools, Kane \cite{Kane}
(generalized by Smith \cite{Smith2},
another independent analytic method by Fouvry-Kl\"uners \cite{FK2})
made a transition from the limit density
to natural density.
For details, one can also see \cite{PT}.
\end{remark}

\subsubsection{}

By extending the above works, we obtain the following positive density results
for tiling number elliptic curve $E:y^2=x(x-1)(x+3)$.

\begin{thm}
\label{positive density 3}
Among the set of positive square-free integers $n\equiv 3\mod 24$,
the limit density of $n$ such that
\begin{itemize}
\item
$\Sel_2(E^{(n)}/\BQ)\big/ E^{(n)}[2]=0$,
\item
$\Sel_2(E^{(-n)}/\BQ)\big/ E^{(-n)}[2]=0$,
\item
both of $\Sel_2(E^{(\pm n)}/\BQ)\big/ E^{(\pm n)}[2]$ are zero,
\end{itemize}
are all equal to $\delta_{\infty,0}
=0.288788\cdots$.
\end{thm}

\begin{proof}
By Proposition \ref{descent main 3},
all of these conditions are equivalent to that $g(n)$ is odd,
namely, $\corank(A+D_{-1})=0$.
The set of positive square-free integers
$n\equiv 3\mod 24$ with $k+1$ prime factors
are one-to-one correspond to the set of positive square-free integers
$n/3\equiv 1\mod 8$ coprime to $3$,
and the element $\omega(n)\in\Omega_{k+1,\{-1\}}/S_{k+1}$
is the image of $\omega(n/3)\in\Omega_{k,\{-1,2,3\}}/S_k$
under the following map:
$$
\psi:\Omega_{k,\{-1,2,3\}}/S_k
=\Omega_{k,\{-1,3\}}/S_k\times\{z_2\in\BF_2^k\}/S_k
\twoheadrightarrow\Omega_{k,\{-1,3\}}/S_k
\xrightarrow{\psi_k}\Omega_{k+1,\{-1\}}/S_{k+1},
$$
where the map $\psi_k$ is defined in Lemma \ref{p:add-primes} below.
Taking $\Omega$ to be the subset of $\Omega_{k+1,\{-1\}}$
such that the sum of elements of $z_{-1}$ equals $1$,
and taking $\Omega_0$ to be the subset of $\Omega$ such that $\corank(A+D_1)=0$.
Then the result of Gerth \cite{Gerth}
just says that $\frac{\#\Omega_0}{\#\Omega}\to\delta_{\infty,0}$
as $k\to\infty$.
Consider their preimages,
utilizing Lemma \ref{p:add-primes},
the desired result follows easily.
\end{proof}

\begin{lem}
\label{p:add-primes}
Let $\Sigma$ be a finite set consisting of $-1$ and some primes of $\BQ$,
$q$ be a prime not in $\Sigma$, and $\Sigma':=\Sigma\cup\{q\}$.
Let $k\geq 1$ be an integer,
$\psi_k:\Omega_{k,\Sigma'}/S_k\to\Omega_{k+1,\Sigma}/S_{k+1}$
be the unique map which makes the following diagram commutes
$$
\xymatrix@C=3em{
~S_k(\Sigma')~\ar[d]\ar@{^(->}[r]^-{n\mapsto qn} & S_{k+1}(\Sigma)\ar[d] \\
\Omega_{k,\Sigma'}/S_k\ar[r]^-{\psi_k} & \Omega_{k+1,\Sigma}/S_{k+1}
}
$$
where $S_k(\Sigma')$ and $S_{k+1}(\Sigma)$
are the set of positive square-free integers
coprime to $\Sigma'$ with $k$ prime factors,
resp. coprime to $\Sigma$ with $k+1$ prime factors.
Let $\Omega_0$ be a subset of $\Omega_{k+1,\Sigma}$ stable by $S_{k+1}$-action.
Then
$$
\left|\frac{\#\psi_k^{-1}(\Omega_0)}{\#\Omega_{k,\Sigma'}}
-\frac{\#\Omega_0}{\#\Omega_{k+1,\Sigma}}\right|\leq 2^{\#\Sigma}\cdot
\frac{\log(k+1)+13/12}{\sqrt{k+1}}.
$$
\end{lem}

\begin{proof}
The map $\psi_k$ can be given by the map
$\widetilde\psi_k:\Omega_{k,\Sigma'}\to\Omega_{k+1,\Sigma}$
which is
$$
\big((a_{ij})_{1\leq i<j\leq k},
(z_i^{(p)})_{1\leq i\leq k,p\in\Sigma'}\big)\mapsto\big((b_{ij})_{1\leq i<j\leq k+1},
(w_i^{(p)})_{1\leq i\leq k+1,p\in\Sigma}\big),
$$
where
$$
b_{ij}=\begin{cases}
a_{ij},&\text{if }1\leq i<j\leq k, \\
z_i^{(q)},&\text{if }1\leq i\leq k\text{ and }j=k+1,
\end{cases}
\qquad\text{and}\qquad
w_i^{(p)}=\begin{cases}
z_i^{(p)},&\text{if }1\leq i\leq k, \\
\lrdd{p}{q},&\text{if }i=k+1.
\end{cases}
$$
It's clear that this map is injective, and whose image is the set
of $\big((b_{ij}),(w_i^{(p)})\big)$ such that
$w_{k+1}^{(p)}=\lrdd{p}{q}$ for all $p\in\Sigma$.
By this way we view $\Omega_{k,\Sigma'}$ as a subset of $\Omega_{k+1,\Sigma}$,
and $\psi_k^{-1}(\Omega_0)=\Omega_0\cap\Omega_{k,\Sigma'}$.

Write $\Omega_{k+1,\Sigma}=\bigsqcup_{s=0}^{k+1}\Omega_{k+1,\Sigma,s}$
where $\Omega_{k+1,\Sigma,s}$ consists of elements
$\big((b_{ij}),(w_i^{(p)})\big)$ of $\Omega_{k+1,\Sigma}$ such that
$\#\{1\leq i\leq k+1\mid w_i^{(p)}=\lrdd{p}{q}\text{ for all }p\in\Sigma\}=s$.
Then $\Omega_{k+1,\Sigma,s}$ is also a union of $S_{k+1}$-orbits, and
$$
\#\Omega_{k+1,\Sigma,s}=2^{k(k+1)/2}\binom{k+1}{s}(2^{\#\Sigma}-1)^{k+1-s}.
$$
It's known that (see for example \cite{MU17}, Theorem 4.12)
for $N\in\BZ_{\geq 1}$, $0<p<1$ and $\delta>0$, we have
$$
\sum_{|n-Np|\geq N\delta}\binom{N}{n}p^n(1-p)^{N-n}\leq
2\exp(-2N\delta^2),
$$
hence
$$
R:=\sum_{|s-(k+1)/2^{\#\Sigma}|\geq(k+1)\delta}
\frac{\#\Omega_{k+1,\Sigma,s}}{\#\Omega_{k+1,\Sigma}}\leq
2\exp(-2(k+1)\delta^2).
$$

On the other hand,
if $\omega$ is an element of $\Omega_{k+1,\Sigma,s}$,
consider the $S_{k+1}$-orbit $[\omega]$, it's easy to see that
$$
\frac{\#([\omega]\cap\Omega_{k,\Sigma'})}{\#[\omega]}
=\frac{s}{k+1}.
$$
Therefore, if we
write $\Omega_{0,s}:=\Omega_0\cap\Omega_{k+1,\Sigma,s}$, then we have
$$
\#(\Omega_0\cap\Omega_{k,\Sigma'})=\sum_{s=0}^{k+1}\#(\Omega_{0,s}\cap\Omega_{k,\Sigma'})
=\sum_{s=0}^{k-1}\frac{s}{k+1}\#\Omega_{0,s}.
$$
Hence
\begin{align*}
\#(\Omega_0\cap\Omega_{k,\Sigma'})
&\leq
\sum_{|s-(k+1)/2^{\#\Sigma}|<(k+1)\delta}
\frac{s}{k+1}\#\Omega_{0,s}
+R\cdot\#\Omega_{k+1,\Sigma} \\
&\leq(2^{-\#\Sigma}+\delta)\cdot\#\Omega_0
+2\exp(-2(k+1)\delta^2)\cdot\#\Omega_{k+1,\Sigma}
\end{align*}
as well as
\begin{align*}
\#(\Omega_0\cap\Omega_{k,\Sigma'})
&\geq
\sum_{|s-(k+1)/2^{\#\Sigma}|<(k+1)\delta}
\frac{s}{k+1}\#\Omega_{0,s}
\geq(2^{-\#\Sigma}-\delta)\cdot(\#\Omega_0-R\cdot\#\Omega_{k+1,\Sigma}) \\
&\geq(2^{-\#\Sigma}-\delta)\cdot\#\Omega_0
-2\exp(-2(k+1)\delta^2)\cdot\#\Omega_{k+1,\Sigma},
\end{align*}
therefore
\begin{align*}
\left|\frac{\#\psi_k^{-1}(\Omega_0)}{\#\Omega_{k,\Sigma'}}
-\frac{\#\Omega_0}{\#\Omega_{k+1,\Sigma}}\right|
&=\left|
\frac{\#(\Omega_0\cap\Omega_{k,\Sigma'})-2^{-\#\Sigma}\cdot\#\Omega_0}{\#\Omega_{k,\Sigma'}}
\right|
\leq\frac{\delta\cdot\#\Omega_0+2\exp(-2(k+1)\delta^2)\cdot\#\Omega_{k+1,\Sigma}}
{\#\Omega_{k,\Sigma'}} \\
&\leq
2^{\#\Sigma}\big(\delta+2\exp(-2(k+1)\delta^2)\big).
\end{align*}
Taking $\delta=\log(k+1)/\sqrt{k+1}$ it's easy to obtain the
desired result.
\end{proof}

\subsubsection{}

The $n\equiv 7\mod 24$ case is more complicated,
since now the condition for $E^{(n)}$
and $E^{(-n)}$ are not the same,
and the combinations of genus numbers $g(n)$ are involved.

First we need a lemma which gives an equivalent condition
of that in Theorem \ref{descent main}.

We introduce some notations. Let $[a\lddot b]$ be the set
$\{a,a+1,\cdots,b\}$, $[k]$ be $[1\lddot k]$,
and if $M=(a_{ij})$ is a matrix,
$I$ and $J$ are subsets of integers, the
$M[I,J]$ is defined to be the submatrix $(a_{ij})_{i\in I,j\in J}$ of $M$.
Define $M_\Rows=(a_{ij}')$ where $a_{ij}'=a_{ij}$ if $i\neq j$,
and $a_{ii}'=-\sum_{j\neq i}a_{ij}$.
For any positive integer $i$, let $e_i$ be the column vector with $i$-th entry $1$,
other entries $0$.

\begin{lem}
\label{linear algebra 7 equation}
Let $n=\ell_1\cdots\ell_k\equiv 7\mod 24$ be positive square-free.
Then we have
$$
\det\begin{pmatrix}
A+D_{-1} & z_{-3} \\
z_2^{\mathrm T} & 1
\end{pmatrix}
\equiv\det\begin{pmatrix}
z_{-1}z_{-2}^{\mathrm T}+z_2z_{-1}^{\mathrm T} & A^{\mathrm T}+D_{-1} \\
A+D_{-1} & D_{-3} \\
\end{pmatrix}\mod 2.
$$
\end{lem}

\begin{proof}
For the simplicity of notation, let $B=z_{-1}z_{-2}^{\mathrm T}+z_2z_{-1}^{\mathrm T}$.
In the right hand side,
by Laplace expansion on rows $k+1$ to $2k$,
we have
$$
\RHS
=\sum_{\substack{I\subset[k]\\
J\subset[k]\\
\#I=\#J}}\det\begin{pmatrix}
(A+D_{-1})\big[[k],[k]\setminus I\big] &
D_{-3}\big[[k],J\big]
\end{pmatrix}\det\begin{pmatrix}
B\big[[k],I\big] &
(A^{\mathrm T}+D_{-1})\big[[k],[k]\setminus J\big]
\end{pmatrix}.
$$
Note that $B$ is of rank at most $2$,
hence the non-zero terms in the above summation should
satisfy $\#I=\#J\leq 2$ and $J\subset\supp(z_{-3})=\{i\mid \ell_i\equiv 2\mod 3\}$,
and in this case, we have
$$
\det\begin{pmatrix}
(A+D_{-1})\big[[k],[k]\setminus I\big] &
D_{-3}\big[[k],J\big]
\end{pmatrix}
=\det\left((A+D_{-1})\big[[k]\setminus J,[k]\setminus I\big]
\right)
$$
and
$$
\det\begin{pmatrix}
B\big[[k],I\big] &
(A^{\mathrm T}+D_{-1})\big[[k],[k]\setminus J\big]
\end{pmatrix}
=\sum_{\substack{I'\subset[k]\\
\#I'=\#I}}
\det\left(B[I',I]\right)
\det\left((A+D_{-1})\big[[k]\setminus J,[k]\setminus I'\big]
\right).
$$
If $I'\neq I$, then we may interchange $I$ and $I'$ in the summation
of $\RHS$,
and they sum to zero
(note that $B=z_{-1}z_2^{\mathrm T}+z_2z_{-1}^{\mathrm T}+z_{-1}z_{-1}^{\mathrm T}$
is a symmetric matrix). Hence we have
$$
\RHS
=\sum_{\substack{I\subset[k]\\
J\subset\supp(z_{-3})\\
\#I=\#J\leq 2}}\det\left(B[I,I]\right)
\det\left((A+D_{-1})\big[[k]\setminus J,[k]\setminus I\big]
\right).
$$
Note that the
$j$-th column of $B$ equals
$$
\begin{cases}
0,&\text{if }\ell_j\equiv 1\mod 8, \\
z_2,&\text{if }\ell_j\equiv 3\mod 8, \\
z_{-1},&\text{if }\ell_j\equiv 5\mod 8, \\
z_{-2},&\text{if }\ell_j\equiv 7\mod 8,
\end{cases}
$$
hence $\det\left(B[I,I]\right)\neq 0$ if and only if $I=\emptyset$,
or $I=\{i\}$ such that $\ell_i\equiv 3\mod 4$,
or $I=\{i_1,i_2\}$ such that $1\not\equiv\ell_{i_1}\not\equiv\ell_{i_2}
\not\equiv 1\mod 8$.
There is only one term in $\RHS$
with $\#I=0$, namely
$$
\det(A+D_{-1})
=\det\begin{pmatrix}
A+D_{-1} & 0 \\
z_2^{\mathrm T} & 1 \\
\end{pmatrix}.
$$
The terms in $\RHS$
with $\#I=1$ sum to
$$
\sum_{\substack{(i,j)\\
\ell_j\equiv 2\mod 3\\
\ell_i\equiv 3\mod 4}}
\det\left((A+D_{-1})\big[[k]\setminus\{j\},[k]\setminus\{i\}\big]
\right)
=\det\begin{pmatrix}
A+D_{-1} & z_{-3} \\
z_{-1}^{\mathrm T} & 0
\end{pmatrix}=0,
$$
note that the rows of $(A+D_{-1},z_{-3})$ sum to $(z_{-1}^{\mathrm T},0)$.
When $\#I=2$, for each $J=\{j_1,j_2\}$, we have
\begin{multline*}
\sum_{\substack{I=\{i_1,i_2\}\\
1\not\equiv\ell_{i_1}\not\equiv\ell_{i_2}
\not\equiv 1\mod 8}}
\det\left((A+D_{-1})\big[[k]\setminus J,[k]\setminus I\big]
\right)
=\det\begin{pmatrix}
A+D_{-1} & e_{j_1} & e_{j_2} \\
z_{-1}^{\mathrm T} & 0 & 0 \\
z_2^{\mathrm T} & 0 & 0 \\
\end{pmatrix} \\
=\det\begin{pmatrix}
A+D_{-1} & e_{j_1} & e_{j_2} \\
0 & 1 & 1 \\
z_2^{\mathrm T} & 0 & 0 \\
\end{pmatrix}
=\det\begin{pmatrix}
A+D_{-1} & e_{j_1}+e_{j_2} \\
z_2^{\mathrm T} & 0 \\
\end{pmatrix},
\end{multline*}
therefore, note that $\supp(z_{-3})$ has even number of elements,
the terms in $\RHS$
with $\#I=2$ sum to
$$
\det\begin{pmatrix}
A+D_{-1} & z_{-3} \\
z_2^{\mathrm T} & 0 \\
\end{pmatrix}.
$$
The desired result follows.
\end{proof}

The above lemma is also valid for $\omega\in\Omega_{k,\{-1,2,3\}}$
such that $\operatorname{sum}(z_{-1})=1$,
$\operatorname{sum}(z_2)=\operatorname{sum}(z_{-3})=0$.

\begin{thm}
\label{positive density 7}
Among the set of positive square-free integers $n\equiv 7\mod 24$,
the limit density of $n$ such that
\begin{itemize}
\item
$\Sel_2(E^{(n)}/\BQ)\big/ E^{(n)}[2]=0$,
\item
$\Sel_2(E^{(-n)}/\BQ)\big/ E^{(-n)}[2]=0$,
\item
both of $\Sel_2(E^{(\pm n)}/\BQ)\big/ E^{(\pm n)}[2]$ are zero,
\end{itemize}
are equal to $\delta_{\infty,0}$,
$\delta_{\infty,0}$, and
$\frac12\delta_{\infty,0}
=0.144394\cdots$,
respectively.
\end{thm}

\begin{proof}
By Theorem \ref{analytic sha main}
and Lemma \ref{linear algebra 7 equation}, the first two
are equivalent to that
$$
\det(A+D_{-1})=1,
\qquad\text{resp.}\qquad
\det\begin{pmatrix}
A+D_{-1} & z_{-3} \\
z_2^{\mathrm T} & 1
\end{pmatrix}=1.
$$
Fix $k\geq 1$, let $\Omega$
be the subset of $\Omega_{k,\{-1,2,3\}}$
such that $\operatorname{sum}(z_{-1})=1$,
$\operatorname{sum}(z_2)=\operatorname{sum}(z_{-3})=0$.
Let $\delta_k^{(1)}$ and $\delta_k^{(2)}$ be the density
of $\omega\in\Omega$ such that the first and the second
of above equation holds, respectively.
Let $\delta_k^{(3)}$ be the density
of $\omega\in\Omega$ such that
$$
\det(A+D_{-1})+\det\begin{pmatrix}
A+D_{-1} & z_{-3} \\
z_2^{\mathrm T} & 1
\end{pmatrix}
=\det\begin{pmatrix}
A+D_{-1} & z_{-3} \\
z_2^{\mathrm T} & 0
\end{pmatrix}
\quad\text{equals}\quad 1.
$$
Let $\delta^{(i)}:=\lim_{k\to\infty}\delta_k^{(i)}$ for $i=1,2,3$.
Then the desired limit density equals
$\delta^{(1)}$, $\delta^{(2)}$,
and $(\delta^{(1)}+\delta^{(2)}-\delta^{(3)})/2$, respectively.

By the same argument in Theorem \ref{positive density 3},
we have $\delta_k^{(1)}=\delta_{k,0}$ and hence $\delta^{(1)}=\delta_{\infty,0}$.
For $\delta_k^{(2)}$ and $\delta_k^{(3)}$, it's easy to see that
$$
\delta_k^{(2)}=\delta_{k,0}Q_{k,0}^{(2)}+\delta_{k,1}Q_{k,1}^{(2)}
\quad\text{and}\quad
\delta_k^{(3)}=\delta_{k,0}Q_{k,0}^{(3)}+\delta_{k,1}Q_{k,1}^{(3)},
$$
where for $j=0,1$, define
\begin{align*}
Q_{k,j}^{(2)}&:=\BP\left(
\det\begin{pmatrix}
A+D_{-1} & z_{-3} \\
z_2^{\mathrm T} & 1
\end{pmatrix}\neq 0
~\middle|
\begin{array}{l}
A,D_{-1}\text{ such that }\corank(A+D_{-1})=j, \\
z_2,z_{-3}\in\BF_2^k\text{ such that }
\operatorname{sum}(z_2)=\operatorname{sum}(z_{-3})=0
\end{array}
\right), \\
Q_{k,j}^{(3)}&:=\BP\left(
\det\begin{pmatrix}
A+D_{-1} & z_{-3} \\
z_2^{\mathrm T} & 0
\end{pmatrix}\neq 0
~\middle|
\begin{array}{l}
A,D_{-1}\text{ such that }\corank(A+D_{-1})=j, \\
z_2,z_{-3}\in\BF_2^k\text{ such that }
\operatorname{sum}(z_2)=\operatorname{sum}(z_{-3})=0
\end{array}
\right).
\end{align*}

For $Q_{k,0}^{(2)}$ and $Q_{k,0}^{(3)}$, it's easy to see that under the
assumption that $A+D_{-1}$ is invertible,
\begin{align*}
Q_{k,0}^{(2)}&=\BP\left(
z_2^{\mathrm T}(A+D_{-1})^{-1}z_{-3}=0
~\middle|
\begin{array}{l}
A,D_{-1}\text{ such that }\corank(A+D_{-1})=0, \\
z_2,z_{-3}\in\BF_2^k\text{ such that }
\operatorname{sum}(z_2)=\operatorname{sum}(z_{-3})=0
\end{array}
\right), \\
Q_{k,0}^{(3)}&=\BP\left(
z_2^{\mathrm T}(A+D_{-1})^{-1}z_{-3}=1
~\middle|
\begin{array}{l}
A,D_{-1}\text{ such that }\corank(A+D_{-1})=0, \\
z_2,z_{-3}\in\BF_2^k\text{ such that }
\operatorname{sum}(z_2)=\operatorname{sum}(z_{-3})=0
\end{array}
\right)
=1-Q_{k,0}^{(2)}.
\end{align*}
Consider $Q_{k,0}^{(2)}$.
Since the $z_2$ should satisfy $\operatorname{sum}(z_2)=0$,
i.e.~$z_2^{\mathrm T}(1,\cdots,1)^{\mathrm T}=0$,
for each $z_{-3}$, the number of choices of $z_2$
such that $z_2^{\mathrm T}(A+D_{-1})^{-1}z_{-3}=0$
equals $2^{k-1}$ if $(A+D_{-1})^{-1}z_{-3}=0$
or $(A+D_{-1})^{-1}z_{-3}=(1,\cdots,1)^{\mathrm T}$,
and equals $2^{k-2}$ in other cases.
Note that $(A+D_{-1})^{-1}z_{-3}$ cannot be $(1,\cdots,1)^{\mathrm T}$,
equivalently, $z_{-3}$ cannot be $(A+D_{-1})(1,\cdots,1)^{\mathrm T}$,
since all the columns of $A+D_{-1}$ sum to $z_{-1}$,
which satisfies $\operatorname{sum}(z_{-1})=1$,
while $\operatorname{sum}(z_{-3})=0$.
Therefore
$$
Q_{k,0}^{(2)}=\frac{2^{k-2}(2^{k-1}-1)+2^{k-1}}{2^{2(k-1)}}
=\frac12+\frac{1}{2^k}\to\frac12
\quad\text{as}\quad k\to\infty,
$$
and $Q_{k,0}^{(3)}=1-Q_{k,0}^{(2)}=\frac12-\frac{1}{2^k}\to\frac12$
as $k\to\infty$.

For $Q_{k,1}^{(2)}$ and $Q_{k,1}^{(3)}$, under the
assumption that $A+D_{-1}$ is of rank $k-1$,
we have $Q_{k,1}^{(2)}=Q_1Q_2$ and $Q_{k,1}^{(3)}=Q_1Q_3$,
where
\begin{align*}
Q_1&:=\BP\left(
\rank(A+D_{-1},z_{-3})=k
~\middle|
\begin{array}{l}
A,D_{-1}\text{ such that }\corank(A+D_{-1})=1, \\
z_{-3}\in\BF_2^k\text{ such that }\operatorname{sum}(z_{-3})=0
\end{array}
\right), \\
Q_2&:=\BP\left(
\rank\begin{pmatrix}
A+D_{-1} & z_{-3} \\
z_2^{\mathrm T} & 1
\end{pmatrix}=k+1
~\middle|
\begin{array}{l}
A,D_{-1}\text{ such that }\corank(A+D_{-1})=1, \\
z_2,z_{-3}\in\BF_2^k\text{ such that }
\operatorname{sum}(z_2)=\operatorname{sum}(z_{-3})=0 \\
\text{and }\rank(A+D_{-1},z_{-3})=k
\end{array}
\right), \\
Q_3&:=\BP\left(
\rank\begin{pmatrix}
A+D_{-1} & z_{-3} \\
z_2^{\mathrm T} & 0
\end{pmatrix}=k+1
~\middle|
\begin{array}{l}
A,D_{-1}\text{ such that }\corank(A+D_{-1})=1, \\
z_2,z_{-3}\in\BF_2^k\text{ such that }
\operatorname{sum}(z_2)=\operatorname{sum}(z_{-3})=0 \\
\text{and }\rank(A+D_{-1},z_{-3})=k
\end{array}
\right).
\end{align*}
Let $V_1\subset\BF_2^k$ be the subspace
of dimension $k-1$ generated by the columns of $A+D_{-1}$,
and $V_2\subset\BF_2^k$ be the subspace
of dimension $k-1$ consisting of $z$ such that $\operatorname{sum}(z)=0$.
Since $z_{-1}\in V_1$ but $z_{-1}\notin V_2$,
the $V_1\cap V_2$ is of dimension $k-2$, and hence we obtain
$$
Q_1=\BP\left(
z_{-3}\in V_2\setminus V_1
~\middle|
\begin{array}{l}
A,D_{-1}\text{ such that }\corank(A+D_{-1})=1, \\
z_{-3}\in\BF_2^k\text{ such that }\operatorname{sum}(z_{-3})=0
\end{array}
\right)=\frac12.
$$
Similarly, let $V_3\subset\BF_2^{k+1}$ be the subspace
of dimension $k$ generated by the rows of $(A+D_{-1},z_{-3})$,
let $V_4\subset\BF_2^{k+1}$ be $V_2\oplus\BF_2$,
and $V_4^0\subset V_4$ be $V_2\oplus 0$,
then $(z_{-1}^{\mathrm T},0)\in V_3\setminus V_4$,
so the $V_3\cap V_4$ and $V_3\cap V_4^0$ are of dimensions $k-1$
and $k-2$, respectively.
Now it's easy to see that
\begin{align*}
Q_2&=\BP\left(
(z_2^{\mathrm T},1)\in(V_4\setminus V_3)
\setminus(V_4^0\setminus V_3)
~\middle|
\begin{array}{l}
A,D_{-1}\text{ such that }\corank(A+D_{-1})=1, \\
z_2,z_{-3}\in\BF_2^k\text{ such that }
\operatorname{sum}(z_2)=\operatorname{sum}(z_{-3})=0 \\
\text{and }\rank(A+D_{-1},z_{-3})=k
\end{array}
\right)=\frac12, \\
Q_3&=\BP\left(
(z_2^{\mathrm T},0)\in V_4^0\setminus V_3
~\middle|
\begin{array}{l}
A,D_{-1}\text{ such that }\corank(A+D_{-1})=1, \\
z_2,z_{-3}\in\BF_2^k\text{ such that }
\operatorname{sum}(z_2)=\operatorname{sum}(z_{-3})=0 \\
\text{and }\rank(A+D_{-1},z_{-3})=k
\end{array}
\right)=\frac12.
\end{align*}
Therefore $Q_{k,1}^{(2)}=Q_{k,1}^{(3)}=1/4$.

The above calculations yields that
$\delta^{(2)}=\delta^{(3)}=\delta_{\infty,0}$,
which gives the desired limit density.
\end{proof}

\begin{remark}
\label{comparison to CNP}
Heath-Brown \cite{HB2} and Swinnerton-Dyer \cite{SD}
proved that, if $E/\BQ$ is an elliptic curve satisfying
$E[2]\subset E(\BQ)$ and $E$ has no cyclic subgroup of order $4$ defined over $\BQ$
(for example $E:y^2=x^3-x$ the congruent number curve),
then for any fixed $S$-equivalence class $\fX$,
the minimum value of the dimension $s(n)$ of the $2$-Selmer group
of $E^{(n)}$ modulo torsion as $n\in\fX$, denoted by $s_{\min}$,
is equal to $0$ or $1$, Moreover, for any $d\in \BZ_{\geq 0}$, the limit density
$$
\limProb\big(s(n)=d\mid n\in\fX\big)
=\begin{cases}
\displaystyle
2\prod_{j=0}^\infty(1+2^{-j})^{-1}\prod_{i=1}^d \frac{2}{2^i-1},&\text{if }
d\equiv s_{\min}\mod 2, \\
0,&\text{if }d\not\equiv s_{\min}\mod 2.
\end{cases}
$$
In particular, if $\fX$ is such that $s_{\min}=0$, then the limit density
$$
\limProb\big(s(n)=0\mid n\in\fX\big)
=2\prod_{j=0}^\infty(1+2^{-j})^{-1}=0.419422\cdots.
$$

Our study reveals some new phenomena on the distribution of
$2$-Selmer groups in the tiling number family
$\pm ny^2=x(x-1)(x+3)$, which does not satisfy the above condition
since it has rational cyclic subgroup of order $4$.
Firstly, there exists $\fX$ such that $s_{\min}\geq 2$
(Proposition \ref{descent non-trivial example}).
Secondly, even if $s_{\min}=0$,
the limit density of the $2$-Selmer modulo torsion being trivial
is different from the above result.
In fact, for all $S$-equivalence classes $\fX$ of $E$
of sign $+1$ (see Lemma \ref{sign}), by the similar method in this section,
we can determine the limit density among $n\in\fX$
such that $s(n)=0$:
$$
\limProb\big(s(n)=0\mid n\in\fX\big)=\begin{cases}
\delta_{\infty,0},
&\text{if }\fX=[n]\text{ with }n\equiv 2,3,5,7,14,15,19\mod 24 \\
&\text{or }\fX=[-n]\text{ with }n\equiv 2,3,6,7,11,14,18\mod 24, \\
\frac{4}{3}\delta_{\infty,0},
&\text{if }\fX=[n]\text{ with }n\equiv 1,9\mod 24, \\
0,&\text{if }\fX=[-n]\text{ with }n\equiv 1\mod 12.
\end{cases}
$$
In particular, the limit density among all square-free integers $n$
of sign $+1$
such that $s(n)=0$
is equal to $\frac{131}{144}\delta_{\infty,0}=0.262716\cdots$.
\end{remark}

\section{Comparison of Selmer and $L$-value}
\label{s:linear-algebra}

In this section we are going to prove

\begin{thm}
\label{comparison 7}
Let $n\equiv 7\mod 24$ be a square-free positive integer.
Then the followings are equivalent:
\begin{itemize}
\item[(a)]
The genus invariant
$$
g(n)+\sum_{d\mid n,d\equiv 11\mod 24}g(n/d)g(d)
$$
is odd.
\item[(b)]
$\Sel_2(E^{(-n)}/\BQ)/E^{(-n)}[2]=0$.
\end{itemize}
\end{thm}

In the following let
$n=\ell_1\cdots \ell_k\neq 1$ be a positive square-free integer coprime to $2$.
We keep the notations $A$, $z_d$ and $D_d$ associated to $n$ as in \S\ref{s:matrices}.

First we recall some well-known results on R\'edei matrix.

\begin{prop}
{\rm(i)}
Let $n\neq 1$, $n\equiv 1\mod 4$ be a positive square-free integer.
Then for any $i$,
$$
g(n)\equiv\det\begin{pmatrix}
A & z_2 \\ e_i^{\mathrm T} & 0
\end{pmatrix}
\equiv\det\begin{pmatrix}
A+D_{-1} & e_i \\ z_2^{\mathrm T}+\left[\frac{2}{n}\right]z_{-1}^{\mathrm T} & 0
\end{pmatrix}\mod 2.
$$

{\rm(ii)}
Let $n\neq 1$, $n\equiv 1\mod 2$ be a positive square-free integer.
Then
$$
g(2n)\equiv\det(A+D_2)\mod 2.
$$

{\rm(iii)}
Let $n\equiv 3\mod 4$ be a positive square-free integer.
Then for any $i$ and $j$,
$$
g(n)\equiv\det\begin{pmatrix}
A & e_i \\ e_j^{\mathrm T} & 0
\end{pmatrix}\equiv\det(A+D_{-1})\mod 2.
$$
\end{prop}

The following is a key linear algebra lemma.

\begin{lem}
\label{linear algebra 3}
Let $k\geq 1$ be an integer,
$n=\ell_1\cdots\ell_k\neq 1$ be a positive square-free integer coprime to $2$.
Then for any $i_1$ and $i_2$,
$$
\sum_{\substack{i_1\in I\subset[k]\\
i_2\notin I}}
\left(\sum_{i\in I'}\left[\frac{-1}{\ell_i}\right]\right)
\det\begin{pmatrix}
A[I',I']_\Rows & e_{i_2}[I'] \\
e_{i_2}^{\mathrm T}[I'] & 0
\end{pmatrix}
\det\begin{pmatrix}
A[I,I]_\Rows & e_{i_1}[I] \\
e_{i_1}^{\mathrm T}[I] & 0
\end{pmatrix}
=\det\begin{pmatrix}
A+D_{-1} & e_{i_1} \\
e_{i_1}^{\mathrm T}+e_{i_2}^{\mathrm T} & 0
\end{pmatrix},
$$
where $I':=[k]\setminus I$.
\end{lem}

\begin{proof}
When $i_1=i_2$ both sides are zero, so after changing the order
of the prime factors of $n$, in the following we may assume $k\geq 2$
and $(i_1,i_2)=(1,2)$.
Then the right hand side
\begin{align*}
\RHS&=\det\begin{pmatrix}
A+z_{-1}z_{-1}^{\mathrm T} & e_1+e_2 \\
e_1^{\mathrm T} & 0
\end{pmatrix}
=\det\begin{pmatrix}
A+z_{-1}z_{-1}^{\mathrm T} & e_1+e_2 & e_2 \\
e_1^{\mathrm T} & 0 & 0 \\
0 & 0 & 1
\end{pmatrix}
=\det\begin{pmatrix}
A & e_1 & e_2 \\
e_1^{\mathrm T} & 0 & 0 \\
z_{-1}^{\mathrm T} & 0 & 0
\end{pmatrix} \\
&=\det\begin{pmatrix}
\left[\frac{-1}{\ell_2}\right] & \left[\frac{-1}{\ell_3}\right] &
\cdots & \left[\frac{-1}{\ell_k}\right] \\
a_{32} & a_{33} & \cdots & a_{3k} \\
\vdots & \vdots & & \vdots \\
a_{k2} & a_{k3} & \cdots & a_{kk} \\
\end{pmatrix}.
\end{align*}
In the left hand side $1\in I$ and $2\in I'$,
so we may set new $I$ and $I'$ be subsets of $[3\lddot k]$
such that the original $I$ and $I'$ equals
$\{1\}\cup I$ and $\{2\}\cup I'$, respectively,
and so
\begin{align*}
\LHS&=\sum_{I\subset[3\lddot k]}
\left(\sum_{i\in\{2\}\cup I'}\left[\frac{-1}{\ell_i}\right]\right)
\left(\sum_{\tau'\in\Aut(I')}
\prod_{\substack{i\in I'\\
\tau'(i)\neq i}}a_{i,\tau'(i)}
\prod_{\substack{i\in I'\\
\tau'(i)=i}}\left(
\sum_{j\in\{2\}\cup I'\setminus\{i\}}a_{ij}\right)\right) \\
&\qquad{}\left(\sum_{\tau\in\Aut(I)}
\prod_{\substack{i\in I\\
\tau(i)\neq i}}a_{i,\tau(i)}
\prod_{\substack{i\in I\\
\tau(i)=i}}\left(
\sum_{j\in\{1\}\cup I\setminus\{i\}}a_{ij}\right)\right).
\end{align*}
Each tuple $(I,I',\tau,\tau')$ appeared in the above sum produces
an element $\sigma=\tau\tau'$ of $\Aut([3\lddot k])$.
Conversely, for an element $\sigma\in\Aut([3\lddot k])$,
let $K\subset[3\lddot k]$ be the set of fixed points of $\sigma$,
and write $[3\lddot k]\setminus K=\bigsqcup_{i=1}^sI_i$
according to the decomposition of $\sigma$
into product of disjoint cycles,
then the set of $(I,I',\tau,\tau')$ which produces $\sigma$
is one-to-one correspondence to the set of $(J,S)$
where $J\subset K$ and $S\subset[s]$,
given by $I=J\sqcup\bigsqcup_{i\in S}I_i$
and hence $I'=(K\setminus J)\sqcup\bigsqcup_{i\in[s]\setminus S}I_i$.
Therefore the $\LHS$ can be rewritten as
$$
\LHS=\sum_{\sigma\in\Aut([3\lddot k])}
\left(\prod_{i\in[3\lddot k]\setminus K}a_{i,\sigma(i)}\right)
\sum_{\substack{J\subset K\\
S\subset[s]}}
I_{\sigma,J,S}
$$
where
$$
I_{\sigma,J,S}:=
\left(\sum_{i\in\{2\}\cup I'}\left[\frac{-1}{\ell_i}\right]\right)
\prod_{i\in K\setminus J}\left(
\sum_{j\in\{2\}\cup I'\setminus\{i\}}a_{ij}\right)
\prod_{i\in J}\left(
\sum_{j\in\{1\}\cup I\setminus\{i\}}a_{ij}\right).
$$
For a fixed $\sigma$ and a fixed $S$,
denote $\widetilde a_{i2}:=\sum_{j\in\{2\}\cup I_0'\cup K\setminus\{i\}}a_{ij}$
for each $i\in K$,
where $I_0':=\bigsqcup_{i\in[s]\setminus S}I_i$.
Also denote $\widetilde a_{-1,2}:=\sum_{i\in\{2\}\cup I_0'\cup K}\left[\frac{-1}{\ell_i}\right]$.
Let $\delta:K\to\BF_2$, $i\mapsto\delta_i$ be such that
$\delta_i=1$ if $i\in J$, $\delta_i=0$ if $i\notin J$,
then it's easy to see that
$$
I_{\sigma,J,S}
=\left(\widetilde a_{-1,2}+\sum_{i\in K}\delta_i\left[\frac{-1}{\ell_i}\right]\right)
\prod_{i\in K}\left(
\widetilde a_{i2}+\sum_{j\in K}\delta_ja_{ij}
\right).
$$
Hence it's easy to see that
\begin{align*}
\sum_{J\subset K}I_{\sigma,J,S}
&=\sum_{\delta:K\to\BF_2}
\left(\widetilde a_{-1,2}+\sum_{i\in K}\delta_i\left[\frac{-1}{\ell_i}\right]\right)
\prod_{i\in K}\left(
\widetilde a_{i2}+\sum_{j\in K}\delta_ja_{ij}
\right) \\
&=\sum_{\phi:\{-1\}\cup K\to\{2\}\cup K}
2^{\#(K\setminus\Im(\phi))}
\prod_{i\in\phi^{-1}(2)}\widetilde a_{i2}
\prod_{i\in\phi^{-1}(K)}a_{i,\phi(i)},
\end{align*}
where for the simplicity of notation,
$a_{-1,j}:=\left[\frac{-1}{\ell_j}\right]$ for each $j\in K$.
Since we are working over $\BF_2$,
for each $\phi$ appeared in the above summation
which contributes a non-zero term,
it must satisfy that $K\subset\Im(\phi)$, and hence
the $\phi^{-1}(2)$ contains at most one element,
therefore
\begin{multline*}
\sum_{\substack{J\subset K\\
S\subset[s]}}
I_{\sigma,J,S}
=2^s\sum_{\phi:\{-1\}\cup K\twoheadrightarrow K}
\prod_{i\in\{-1\}\cup K}a_{i,\phi(i)} \\
+\sum_{\phi:\{-1\}\cup K\twoheadrightarrow\{2\}\cup K}
\left(2^s\sum_{j\in\{2\}\cup K\setminus\{\phi^{-1}(2)\}}a_{\phi^{-1}(2),j}
+2^{s-1}\sum_{j\in[3\lddot k]\setminus K}a_{\phi^{-1}(2),j}\right)
\prod_{i\in\phi^{-1}(K)}a_{i,\phi(i)},
\end{multline*}
where $2^{s-1}$ is understood as zero if $s=0$.
In other words,
$$
\sum_{\substack{J\subset K\\
S\subset[s]}}
I_{\sigma,J,S}
=\begin{cases}
\displaystyle\sum_{\phi:\{-1\}\cup K\twoheadrightarrow K}
\prod_{i\in\{-1\}\cup K}a_{i,\phi(i)} \\
\displaystyle\qquad{}+\sum_{\phi:\{-1\}\cup K\twoheadrightarrow\{2\}\cup K}
\left(\sum_{j\in\{2\}\cup K\setminus\{\phi^{-1}(2)\}}a_{\phi^{-1}(2),j}
\right)
\prod_{i\in\phi^{-1}(K)}a_{i,\phi(i)},&\text{if }s=0, \\
\displaystyle\sum_{\phi:\{-1\}\cup K\twoheadrightarrow\{2\}\cup K}
\left(\sum_{j\in[3\lddot k]\setminus K}a_{\phi^{-1}(2),j}\right)
\prod_{i\in\phi^{-1}(K)}a_{i,\phi(i)}, &\text{if }s=1, \\
0,&\text{if }s\geq 2.
\end{cases}
$$
Note that $s=0$ if and only if $\sigma=\id$, and in this case $K=[3\lddot k]$.
Therefore we may write $\LHS=\LHS_1+\LHS_2+\LHS_3$ with
\begin{align*}
\LHS_1&=\sum_{\phi:\{-1\}\cup[3\lddot k]\twoheadrightarrow[3\lddot k]}
\prod_{i\in\{-1\}\cup[3\lddot k]}a_{i,\phi(i)}, \\
\LHS_2&=\sum_{\phi:\{-1\}\cup[3\lddot k]\twoheadrightarrow[2\lddot k]}
\left(\sum_{j\in[2\lddot k]\setminus\{\phi^{-1}(2)\}}a_{\phi^{-1}(2),j}
\right)
\prod_{i\in\phi^{-1}([3\lddot k])}a_{i,\phi(i)}, \\
\LHS_3&=\sum_{\substack{\sigma\in\Aut([3\lddot k])\\
\sigma\neq 1\text{ cycle}}}
\left(\prod_{i\in[3\lddot k]\setminus K}a_{i,\sigma(i)}\right)
\sum_{\phi:\{-1\}\cup K\twoheadrightarrow\{2\}\cup K}
\left(\sum_{j\in[3\lddot k]\setminus K}a_{\phi^{-1}(2),j}\right)
\prod_{i\in\phi^{-1}(K)}a_{i,\phi(i)}.
\end{align*}
It's easy to see that
$$
\LHS_2=\det\begin{pmatrix}
\left[\frac{-1}{\ell_2}\right] & \left[\frac{-1}{\ell_3}\right] &
\cdots & \left[\frac{-1}{\ell_k}\right] \\
a_{32}+a_{33} & a_{33} & \cdots & a_{3k} \\
\vdots & \vdots & & \vdots \\
a_{k2}+a_{kk} & a_{k3} & \cdots & a_{kk} \\
\end{pmatrix},
$$
as well as
$$
\LHS_{1\mathrm a}
:=\sum_{\substack{\phi:\{-1\}\cup[3\lddot k]\twoheadrightarrow[3\lddot k]\\
\text{such that let }j\in[3\lddot k]\text{ be the unique element}\\
\text{such that }\#\phi^{-1}(j)=2,\text{ then }j\in\phi^{-1}(j)}}
\prod_{i\in\{-1\}\cup[3\lddot k]}a_{i,\phi(i)}
=\det\begin{pmatrix}
0 & \left[\frac{-1}{\ell_3}\right] &
\cdots & \left[\frac{-1}{\ell_k}\right] \\
a_{33} & a_{33} & \cdots & a_{3k} \\
\vdots & \vdots & & \vdots \\
a_{kk} & a_{k3} & \cdots & a_{kk} \\
\end{pmatrix},
$$
which is a part of $\LHS_1$. Therefore in order to show
$\LHS=\RHS$, we only need to prove that
$$
\LHS_{1\mathrm b}
:=\sum_{\substack{\phi:\{-1\}\cup[3\lddot k]\twoheadrightarrow[3\lddot k]\\
\text{such that let }j\in[3\lddot k]\text{ be the unique element}\\
\text{such that }\#\phi^{-1}(j)=2,\text{ then }j\notin\phi^{-1}(j)}}
\prod_{i\in\{-1\}\cup[3\lddot k]}a_{i,\phi(i)}
=\LHS_3.
$$
For each non-trivial cycle $\sigma\in\Aut([3\lddot k])$,
each $j\in[3\lddot k]\setminus K$ and
each $\phi:\{-1\}\cup K\twoheadrightarrow\{2\}\cup K$
appeared in $\LHS_3$, we can define a map $\phi':\{-1\}\cup[3\lddot k]\to
[3\lddot k]$ by
$$
\phi'(i):=\begin{cases}
\phi(i),&\text{if }i\in\phi^{-1}(K), \\
j,&\text{if }i=\phi^{-1}(2), \\
\sigma(i),&\text{if }i\in[3\lddot k]\setminus K,
\end{cases}
$$
then it's easy to see that $\phi'$ is surjective, $j$ is the unique element
in $[3\lddot k]$ such that $\#(\phi')^{-1}(j)=2$,
and $(\phi')^{-1}(j)=\{\phi^{-1}(2),\sigma^{-1}(j)\}$ which doesn't contain $j$.
Conversely, if $\phi':\{-1\}\cup[3\lddot k]\twoheadrightarrow
[3\lddot k]$ is appeared in $\LHS_{1\mathrm b}$,
namely, the unique element $j\in[3\lddot k]$
such that $\#(\phi')^{-1}(j)=2$ satisfies $j\notin(\phi')^{-1}(j)$,
then the set $\{(\phi')^{(m)}(j)\}_{m\geq 0}$ has at least two elements,
and by the property of $j$, $(\phi')^{(m)}(j)\mapsto(\phi')^{(m+1)}(j)$ gives a non-trivial
cycle $\sigma\in\Aut([3\lddot k])$
with $j\in[3\lddot k]\setminus K=\{(\phi')^{(m)}(j)\}_{m\geq 0}$,
and we can define $\phi:\{-1\}\cup K\to\{2\}\cup K$
by
$$
\phi(i):=\begin{cases}
\phi'(i),&\text{if }i\in(\phi')^{-1}(K), \\
2,&\text{if }i\in(\phi')^{-1}(j),
\end{cases}
$$
and it's easy to see that $\phi$ is surjective.
Now it's easy to see that these two assignments are inverse to each other,
hence we deduce that $\LHS_{1\mathrm b}=\LHS_3$, which completes the proof.
\end{proof}

\begin{lem}
\label{linear algebra 2}
Let $n\equiv 7\mod 8$ be a positive square-free integer.
Then for any $t$,
$$
\sum_{\substack{\ell_t\mid d\mid n\\
d\equiv 5\mod 8}}g(n/d)g(d)
\equiv\det\begin{pmatrix}
A+D_{-1} & e_t \\
z_2^{\mathrm T} & 0
\end{pmatrix}\mod 2.
$$
\end{lem}

\begin{proof}
Let $k$ be the number of prime factors
of $n$. We may assume that $k\geq 2$.
By changing the order of the prime factors
$\ell_1,\cdots,\ell_k$ of $n$, we only need to prove the $t=1$ case.
In this case the left hand side can be written as
$$
\LHS=\sum_{1\in I\subset[k]}
\left(\sum_{i\in I'}\left[\frac{-1}{\ell_i}\right]\right)
\left(\sum_{i\in I'}\left[\frac{2}{\ell_i}\right]\right)
\det\begin{pmatrix}
A[I',I']_\Rows & e_{i'}[I'] \\
e_{i'}^{\mathrm T}[I'] & 0
\end{pmatrix}
\det\begin{pmatrix}
A[I,I]_\Rows & z_2[I] \\
e_1^{\mathrm T}[I] & 0
\end{pmatrix},
$$
where $I':=[k]\setminus I$,
and for each $I$ such that $I'\neq\emptyset$,
we fix an element $i'\in I'$.
The right hand side can be written sa
$$
\RHS=\sum_{i=1}^k\left[\frac{2}{\ell_i}\right]
\det\begin{pmatrix}
A+D_{-1} & e_1 \\
e_i^{\mathrm T} & 0
\end{pmatrix}
=\sum_{i=2}^k\left[\frac{2}{\ell_i}\right]
\det\begin{pmatrix}
A+D_{-1} & e_1 \\
e_1^{\mathrm T}+e_i^{\mathrm T} & 0
\end{pmatrix}.
$$
On the other hand, if $1\in I$, then
$$
\det\begin{pmatrix}
A[I,I]_\Rows & z_2[I] \\
e_1^{\mathrm T}[I] & 0
\end{pmatrix}
=\sum_{i=2}^k\left[\frac{2}{\ell_i}\right]
\det\begin{pmatrix}
A[I,I]_\Rows & (e_1+e_i)[I] \\
e_1^{\mathrm T}[I] & 0
\end{pmatrix}.
$$
Now we regard the entries in $z_2$ be independent of those in $A$ and $z_{-1}$,
namely, view $\left[\frac{2}{\ell_2}\right],\cdots,\left[\frac{2}{\ell_k}\right]$
as independent variables in $\BF_2$, and let
$\left[\frac{2}{\ell_1}\right]
:=\left[\frac{2}{\ell_2}\right]+\cdots+\left[\frac{2}{\ell_k}\right]$.
It's easy to see that if $2\leq i_1<i_2\leq k$,
then the coefficient of $\left[\frac{2}{\ell_{i_1}}\right]
\left[\frac{2}{\ell_{i_2}}\right]$ in $\LHS$ is zero,
and now we only need to prove that for each $2\leq j\leq k$,
the coefficient of $\left[\frac{2}{\ell_j}\right]^2
=\left[\frac{2}{\ell_j}\right]$ in $\LHS$ and $\RHS$ are equal,
namely,
$$
\sum_{\substack{1\in I\subset[k]\\
j\notin I}}
\left(\sum_{i\in I'}\left[\frac{-1}{\ell_i}\right]\right)
\det\begin{pmatrix}
A[I',I']_\Rows & e_j[I'] \\
e_j^{\mathrm T}[I'] & 0
\end{pmatrix}
\det\begin{pmatrix}
A[I,I]_\Rows & e_1[I] \\
e_1^{\mathrm T}[I] & 0
\end{pmatrix}
=\det\begin{pmatrix}
A+D_{-1} & e_1 \\
e_1^{\mathrm T}+e_j^{\mathrm T} & 0
\end{pmatrix}.
$$
This is implied by Lemma \ref{linear algebra 3}.
\end{proof}

\begin{proof}[Proof of Theorem \ref{comparison 7}]
By Theorem \ref{descent main}
and Lemma \ref{linear algebra 7 equation},
we only need to show that
for any positive square-free integer
$n\equiv 7\mod 24$,
$$
\sum_{\substack{d\mid n\\
d\equiv 11\mod 24}}g(n/d)g(d)
\equiv
\det\begin{pmatrix}
A+D_{-1} & z_{-3} \\
z_2^{\mathrm T} & 0 \\
\end{pmatrix}\mod 2.
$$
This is true, since by Lemma \ref{linear algebra 2},
the right hand side is equal to
$$
\sum_{\substack{i\text{ such that}\\
\ell_i\equiv 2\mod 3}}
\sum_{\substack{\ell_i\mid d\mid n\\
d\equiv 5\mod 8}}g(n/d)g(d),
$$
and if $d\mid n$
is such that $d\equiv 5\mod 8$, then $d\equiv 5\mod 24$ or $d\equiv 13\mod 24$,
which have odd or even number of prime factors $\equiv 2\mod 3$, respectively,
and the latter one sums to zero in the summation.
Hence this is also equal to the left hand side.
\end{proof}

\appendix

\section{Some basic facts}

Let $E$ be an elliptic curve over $\BQ$.  Its minimal Werierstrass equation, say of discriminant $\Delta$, gives its minimal regular model $\CE$ over $\BZ$. For each prime $p$,  the local Tamagawa number $c_p$
is the number of irreducible components  of the special fiber $\CE_{\ov{\BF}_p}$,
which can be computed using Tate's algorithm \cite{Si}.  The conductor $N$ of $E$ satisfies the Ogg's formula \cite{Si}: $\ord_p(N)=\ord_p(D)+1-c_p$ for any prime $p$.

Recall that the Manin constant $c_E$ of $E$  is the positive integer such that $\phi^*\omega_E
=\pm c_E \omega_f$, where $\omega_E$ is a N\'eron differential
associated to a global minimal Weierstrass equation
of $E$, $\omega_f=2\pi i f(z) dz$ with $f$ the newform associated to $E$, and $\phi$ is a parametrization $\phi: X_0(N)\ra E$ with minimal degree.
Define the real period $\Omega(E)$
and the imaginary period $\Omega^-(E)$ of $E$ by
$$
\Omega(E):=\int_{E(\BR)}\omega_E,\qquad
\Omega^-(E):=\int_{\gamma^-}\omega_E,
$$
where $\gamma^-$ is a generator of $H_1(E(\BC),\BZ)^-$.
We choose $\omega_E$ and $\gamma^-$ such that $\Omega(E)\in\BR_{>0}$
and $\Omega^-(E)\in i\BR_{>0}$.
Let $\Lambda_E:=\left\{\int_\gamma\omega_E
~\middle|~\gamma\in H_1(E(\BC),\BZ)\right\}\subset\BC$
be the period lattice of $E$, then it's easy to see that
\begin{equation}
\label{period of mod form and ec}
8\pi^2(f,f)_{\Gamma_0(N)}=\frac{\deg\phi}{c_E^2}\cdot 2\operatorname{Vol}(\BC/\Lambda_E)
=\frac{\deg\phi}{c_E^2}\cdot\frac{\Omega(E)\Omega^-(E)}{i}.
\end{equation}
There are  algorithms to compute the modular degree $\deg\phi$ \cite{Cr1}, \cite{Cr2} and the Manin constant $c_E$ (see Appendix of \cite{ARS}).
For the tiling number elliptic curve $y^2=x(x-1)(x+3)$, we have:

\begin{lem}
\label{mod-degree}
Let $E$ be the elliptic curve $y^2=x(x-1)(x+3)$. Then the conductor of $E$ is $24$,
the modular degree of $E$ is $1$ and the Manin constant of $E$ is $1$; The conductor of $E^{(-1)}$ is $48$, its modular degree  is $2$ and the Manin constant  is $1$.
\end{lem}

The following result gives the behavior of various arithmetic invariants
in a quadratic twist family of elliptic curves.

\begin{lem}
\label{quad twist arith values}
Let $E$ be an elliptic curve over $\BQ$ of conductor $N$.
Let $S$ be the set of places of $\BQ$ containing $2,\infty$ and bad places of $E$.
Let $n$ be any square-free integer.

{\rm(i)}
For each $p$, the following values only depend on the image
of $n$ in $\BQ_p^\times/(\BQ_p^\times)^2$:
\begin{itemize}
\item[(a)]
The local root number $\epsilon_p(E^{(n)})$.
\item[(b)]
The local Tamagawa number $c_p(E^{(n)})$.
\end{itemize}

{\rm(ii)}
The following values only
depend on the $S$-equivalence class $[n]$ containing $n$:
\begin{itemize}
\item[(a)]
The value $\sign(n)\epsilon(E^{(n)})$.
\item[(b)]
The value $4^{-\mu(n)}\prod_p c_p(E^{(n)})$.
Here for simplicity we assume that $E[2]\subset E(\BQ)$.
\item[(c)]
The value $u^{(n)}$ which is a positive rational number
such that $n^6\Delta=(u^{(n)})^{12}\Delta^{(n)}$,
where $\Delta$ and $\Delta^{(n)}$ are the discriminants of
minimal Weierstrass equations of $E$ and $E^{(n)}$, respectively.
The $u^{(n)}$ also gives the relation of periods of $E$ under quadratic twists:
$$
u^{(n)}=\begin{cases}
\displaystyle
\frac{\sqrt{n}\cdot\Omega(E^{(n)})}{\Omega(E)}
=\frac{\sqrt{n}\cdot\Omega^-(E^{(n)})}{\Omega^-(E)},
&\text{if }n>0, \\[0.5em]
\displaystyle
\frac{\sqrt{n}\cdot\Omega(E^{(n)})}{c_\infty(E)\Omega^-(E)}
=\frac{-c_\infty(E)\sqrt{n}\cdot\Omega^-(E^{(n)})}{\Omega(E)}
&\text{if }n<0,
\end{cases}
$$
where $c_\infty(E)$ is the number of connected components of $E(\BR)$.
\end{itemize}
\end{lem}

\begin{proof}
(i)
Let $\ell\neq p$ be a prime, $\chi:G_\BQ\to\{\pm 1\}$ be the quadratic
character associated to $\BQ(\sqrt{n})/\BQ$.
Then all of these values only depend on the isomorphism class
of the $\ell$-adic $G_{\BQ_p}$ representation
$(T_\ell(E)\otimes\chi)|_{G_{\BQ_p}}
=T_\ell(E)|_{G_{\BQ_p}}\otimes\chi|_{G_{\BQ_p}}$.
Since the isomorphism class of $\chi|_{G_{\BQ_p}}$
only depends on the image
of $n$ in $\BQ_p^\times/(\BQ_p^\times)^2$, the desired result follows.

(ii) By \cite{Hal} we know that if $p\nmid 2N$, then
$\epsilon_p(E^{(n)})=1$ if $p\nmid n$,
$\epsilon_p(E^{(n)})=\left(\frac{-1}{p}\right)$ if $p\mid n$.
Since $\prod_{p\mid 2N}\epsilon_p(E^{(n)})$ only depends
on $[n]$,
and $\sign(n)\prod_{p\nmid 2N}\epsilon_p(E^{(n)})
=\sign(n)(-1)^{(n'-1)/2}
=(-1)^{(n''-1)/2}$ also only depends
on $[n]$,
here $n'\geq 1$ is the prime-to-$2N$ part of $|n|$
and $n'':=\sign(n)\cdot n'$,
the result (a) follows.

Under the assumption that $E[2]\subset E(\BQ)$, by Tate's algorithm
we know that if $p\nmid 2N$, then
$c_p(E^{(n)})=1$ if $p\nmid n$,
$c_p(E^{(n)})=4$ if $p\mid n$.
Since $\prod_{p\mid 2N}c_p(E^{(n)})$ only depends
on $[n]$,
and $4^{-\mu(n)}\prod_{p\nmid 2N}c_p(E^{(n)})
=4^{-\mu(n)}4^{\mu(n')}=4^{-\mu(n/n')}$ also only depends
on $[n]$,
the result (b) follows.

The result (c) is \cite{Pal}, Proposition 2.5.
\end{proof}

In the following, we collect the information of root number, N\'eron period, and local Tamagawa numbers of $E^{(n)}:y^2=x(x-n)(x+3n)$.

First we point out a simple fact: for any square-free integer $n$,
$y^2=x(x-n)(x+3n)$ is a global minimal Weierstrass
equation for $E^{(n)}$.
This follows from Proposition 2.5 of \cite{Pal} or by a direct
calculation.

\begin{lem}\label{torsion}  Let $n$ be a square-free integer, then we have that $E^{(n)}(\BQ)_\tor=E^{(n)}[2]$ except when $n=1$, $E(\BQ)=\{O, (0, 0), (1, 0), (-3, 0), (-1, \pm 2), (3, \pm 6)\}\cong \BZ/2\BZ \times \BZ/4\BZ$.
\end{lem}
\begin{proof}
Since $E^{(n)}[2]=\{O, (0, 0), (n, 0), (-3n, 0)\}\subset E^{(n)}(\BQ)_\tor$,
by Mazur's classification of torsion points
of elliptic curves over $\BQ$, we have $E^{(n)}(\BQ)_\tor\cong\BZ/2\BZ\times
\BZ/2m\BZ$ for some $m\in\{1,2,3,4\}$.
Hence we only need to show that $E^{(n)}$ has no rational $3$-torsions or $8$-torsions,
and has rational $4$-torsions if and only if $n=1$.

To determine $4$-torsions and $8$-torsions,
for any point $P=(x,y)\in E^{(n)}$, we have
$$
x([2]P)=\frac{(x^2+3n^2)^2}{4x(x-n)(x+3n)}
=\left(\frac{x^2+3n^2}{2y}\right)^2.
$$
Hence if $P=(x,y)$ is a rational primitive $4$-torsion of $E^{(n)}(\BQ)$, then
$x([2]P)=\left(\frac{x^2+3n^2}{2y}\right)^2\in\{0,n,-3n\}$.
If $x([2]P)=0$, then $0=x^2+3n^2\geq 3n^2\geq 3$, contradiction.
If $x([2]P)=n$, then we must have $n=1$, and $P=(-1,\pm 2)$ or $(3,\pm 6)$.
If $x([2]P)=-3n$, then we must have $n=-3$, however
$(x^2+27)^2/4x(x+3)(x-9)=9$ does not have rational solutions, contradiction.

Suppose
$P=(x,y)$ is a rational primitive $8$-torsion of $E^{(n)}(\BQ)$, then
$n=1$ and
$x([2]P)=(x^2+3)^2/4x(x-1)(x+3)\in\{-1,3\}$.
They don't have rational solutions, contradiction.

Suppose
$P=(x,y)$ is a rational primitive $3$-torsion of $E^{(n)}(\BQ)$, then
$x\notin\{0,n,-3n\}$ and
$x([2]P)=(x^2+3n^2)^2/4x(x-n)(x+3n)=x$,
hence $3(\frac xn)^4+8(\frac xn)^3-18(\frac xn)^2-9=0$,
which doesn't have rational solutions, contradiction.
\end{proof}

\begin{lem} \label{sign} Let $n$ be a square-free positive integer and let $\epsilon(E^{(n)})$ denote the root number of $E^{(n)}:y^2=x(x-n)(x+3n)$. Then we have
$$
\begin{tabular}{|c|ccc|ccc|ccc|ccc|ccc|ccc|}
\hline
	$n\bmod 24$ & $1$ & $2$ & $3$ & $5$ & $6$ & $7$
	 & $9$ & $10$ & $11$ & $13$ & $14$ & $15$ & $17$ & $18$ & $19$
	 & $21$ & $22$ & $23$\\
\hline
$\epsilon(E^{(n)})$ & $1$ & $1$ & $1$ & $1$ & $-1$ & $1$
 & $1$ & $-1$ & $-1$&$-1$ & $1$ & $1$ & $-1$ & $-1$ & $1$
 & $-1$ & $-1$ & $-1$ \\
\hline
$\epsilon(E^{(-n)})$ & $1$ & $1$ & $1$ & $-1$ & $1$ & $1$
 & $-1$ & $-1$ & $1$& $1$ & $1$ & $-1$ & $-1$ & $1$ & $-1$
 & $-1$ & $-1$ & $-1$ \\
\hline
\end{tabular}
$$
In particular, the root number of $E^{(n)}$  with $n$ square-free only depends on the sign of $n$ and its residue class modulo $24$.
\end{lem}
\begin{proof}
Let $n$ be a square-free integer, follow from \cite{Hal}
we have $\epsilon(E^{(n)})=-\epsilon_2\epsilon_3\cdot(-1)^{\frac{n'-1}2}$, where
$$
\epsilon_2=\begin{cases}
1,&\text{if }n\equiv 2,3,5,7\pmod{8}, \\
-1,&\text{if }n\equiv 1,6\pmod{8},
\end{cases}
\qquad
\epsilon_3=\begin{cases}
1,&\text{if }n\equiv 1\pmod{3}, \\
-1,&\text{if }n\equiv 0,2\pmod{3},
\end{cases}
$$
and $n'\in\BZ_{\geq 1}$ is the prime-to-$6$ part of $|n|$.
\end{proof}

\begin{lem}
\label{period-quad-twist}
Let $n$ be a square-free integer. Then the periods
of the quadratic twist $E^{(n)}$ of $E:y^2=x(x-1)(x+3)$ is
$$
\Omega(E^{(n)})=\begin{cases}
\Omega(E)/\sqrt{n},&\text{if }n>0, \\
2\Omega^-(E)/\sqrt{n},&\text{if }n<0,
\end{cases}
\quad\text{and}\quad
\Omega^-(E^{(n)})=\begin{cases}
\Omega^-(E)/\sqrt{n},&\text{if }n>0, \\
-\Omega(E)/2\sqrt{n},&\text{if }n<0.
\end{cases}
$$
\end{lem}

\begin{proof}
This follows from Main Result 1.1 of \cite{Pal},
or Lemma \ref{quad twist arith values}
together with the fact that
$y^2=x(x-n)(x+3n)$ is a global minimal Weierstrass
equation for $E^{(n)}$.
\end{proof}

\begin{lem}\label{Tamagawa}
Let $n$ be a square-free integer. Then the Tamagawa number
for $E^{(n)}:y^2=x(x-n)(x+3n)$ is
$$
c_2=\begin{cases}
2,&\text{if }n\equiv 3,5,7\pmod{8}, \\
4,&\text{if }n\equiv 1,2,6\pmod{8},
\end{cases}
\qquad
c_3=\begin{cases}
2,&\text{if }3\nmid n, \\
4,&\text{if }3\mid n,
\end{cases}
$$
and $c_p=4$ if $p\geq 5$ and $p\mid n$.
\end{lem}

\begin{proof}
The curve $E^{(n)}$ has split (resp.~non-split, potential split) multiplicative reduction at the place $3$ if $n\equiv 2\mod 3$ (resp. $n\equiv 1\mod 3$, $n\equiv 0\mod 3$),
which gives result of $c_3$.
It has potential good reduction at $p$ if $p\geq 5$
and $p\mid n$, which gives $c_p$.
It has potential good reduction at $2$, and the result of $c_2$
follows from Tate's algorithm (or by \cite{Hal}).
\end{proof}

\end{document}